\documentclass[a4paper,reqno]{amsart}
\usepackage[all,cmtip]{xy}
\usepackage{stmaryrd,amssymb,mathtools,textcmds,amscd,url}
\usepackage[inline,shortlabels]{enumitem}

\newtheorem{theorem}{Theorem}[section]
\newtheorem{proposition}[theorem]{Proposition}
\newtheorem{lemma}[theorem]{Lemma}

\theoremstyle{definition}
\newtheorem{definition}[theorem]{Definition}
\newtheorem{example}[theorem]{Example}
\theoremstyle{remark}
\newtheorem{remark}[theorem]{Remark}
\newtheorem*{remark*}{Remark}
\numberwithin{equation}{section}

\newcommand{\cV}{{\mathcal V}}
\newcommand{\cH}{{\mathcal H}}
\newcommand{\cE}{{\mathcal E}}
\newcommand{\G}{\mathit{G}}
\newcommand{\C}{{\mathbb C}}
\newcommand{\bR}{{\mathbb R}}
\newcommand{\bZ}{{\mathbb Z}}

\newcommand{\lie}[1]{\mathfrak{#1}}

\newcommand{\mfe}{\lie{e}}
\newcommand{\mfg}{\lie{g}}
\newcommand{\mfh}{\lie{h}}

\newcommand{\mfp}{\lie{p}}

\newcommand{\mfu}{\lie{u}}

\newcommand{\sL}{\lie{sl}}
\newcommand{\su}{\lie{su}}

\newcommand{\spa}[1]{\mathrm{span}(#1)}

\newcommand{\vol}{\mathrm{vol}}

\newcommand{\mc}{\mathcal}
\newcommand{\mf}{\mathfrak}

\newcommand{\pdiff}[2]{\frac{\partial #1}{\partial #2}}

\DeclareMathOperator{\End}{End}

\DeclareMathOperator{\ad}{ad}

\DeclareMathOperator{\diag}{diag}
\DeclareMathOperator{\id}{id}

\DeclareMathOperator{\tr}{tr}
\DeclareMathOperator{\SU}{SU}

\setlist{nosep}

\newcommand{\R}{\mathbb{R}}
\newcommand{\GL}{\mathrm{GL}}
\newcommand{\SL}{\mathrm{SL}}
\newcommand{\SO}{\mathrm{SO}}
\newcommand{\so}{\mathfrak{so}}
\newcommand{\spin}{\mathfrak{spin}}
\newcommand{\Spin}{\mathrm{Spin}}
\newcommand{\Sp}{\mathrm{Sp}}
\newcommand{\Sym}{\operatorname{Sym}}
\newcommand{\divgc}{\operatorname{div}}
\newcommand{\grad}{\operatorname{grad}}
\newcommand{\Diff}{\operatorname{Diff}}
\renewcommand{\Re}{\mathrm{Re}}
\renewcommand{\Im}{\mathrm{Im}}

\begin{document}

\title{Homogeneous spinor flow}
\author{Marco Freibert}
\address{Mathematisches Seminar der Universit\"at Kiel, Ludewig-Meyn-Stra{\ss}e 4, D--24098 Kiel, Germany}
\email{freibert@math.uni-kiel.de}
\address{Department of Mathematics, King’s College London, Strand, London WC2R 2LS, United Kingdom}
\email{marco.freibert@kcl.ac.uk}
\author{Lothar Schiemanowski}
\address{Mathematisches Seminar der Universit\"at Kiel\\ Ludewig-Meyn-Stra{\ss}e 4\\ D--24098 Kiel\\ Germany}
\email{schiemanowski@math.uni-kiel.de}
\author{Hartmut Wei\ss}
\address{Mathematisches Seminar der Universit\"at Kiel\\ Ludewig-Meyn-Stra{\ss}e 4\\ D--24098 Kiel\\ Germany}
\email{weiss@math.uni-kiel.de}

\begin{abstract}
We study the spinor flow on homogeneous spin manifolds. After providing the general setup we discuss the homogeneous spinor flow in dimension 3 and on almost abelian Lie groups in detail. As a further example the flag manifold in dimension 6 is treated.
\end{abstract}

\maketitle

\section{Introduction}
The spinor flow is a geometric evolution equation for a pair consisting of a Riemannian metric $g$ and a unit spinor $\varphi$ on a spin manifold $M$. Here a unit spinor is a section of unit length of the complex spinor bundle $\Sigma_gM$ determined by the spin structure and the Riemannian metric $g$. It is the negative gradient flow of the spinorial energy functional
\[
\mathcal{E}(g,\varphi) = \frac 12 \int_M | \nabla^g \varphi|^2 \, \vol_g
\]
on the set
$$\mathcal{N} = \{(g, \varphi): g \in \Gamma(\odot^2_+ T^*M), \varphi \in \Gamma(\Sigma_g M), |\varphi|=1\}.$$
The critical points of this flow in dimension $3$ and up are absolute minimizers, i.e.\@ pairs of Riemannian metrics $g$ and unit spinor fields $\varphi$ satisfying $\nabla^g \varphi \equiv 0$. This implies that $g$ is a Ricci flat metric of special holonomy. Moreover,
pairs of a Riemannian metric and a unit Killing spinor field are particular cases of volume constrained critical points of the spinorial energy. Thus, the spinor flow is a natural tool to study the geometry of special holonomy spaces and the related weak holonomy spaces.

Short-time existence of the spinor flow on a compact spin manifold has been established in \cite{aww1} and first steps towards understanding the long-time behaviour have been taken in \cite{schie1}, \cite{schie2}. In particular it is shown in \cite{schie1} that the flow is stable near a critical point, i.e.\ a metric $g$ together with a parallel spinor $\varphi$ (in this article we will restrict to dimensions $n \geq 3$, the surface case has been addressed in \cite{aww2}). The behaviour of the volume-normalized spinor flow near a constrained critical point (e.g.\@ a metric $g$ together with a Killing spinor $\varphi$) is more subtle, see again \cite{schie1}. The spinor flow equation is a quasilinear parabolic equation, thus understanding the global behavior of the spinorial energy and the spinor flow equation is a difficult problem in general. Understanding their behaviour in the more restricted homogeneous setting has been one of the motivations for this present article.


A fact of fundamental importance for this study is that the spinorial energy is spin diffeomorphism invariant. This invariance implies that the spinor flow preserves symmetries, in the sense that if $F^*(g,\varphi) = (g, \varphi)$, then $F^*(g_t, \varphi_t) = (g_t, \varphi_t)$ for the solution of the spinor flow with initial condition $(g, \varphi)$. This suggests that we can define a {\em homogeneous spinor flow} for initial conditions, which are invariant under a transitive group action $G \curvearrowright M$. The domain $\mc{N}$ of the spinorial energy can be considered as the space of sections of the so-called \emph{universal spinor bundle}, which fits the space of metrics and the associated spinor bundles into a single fiber bundle. The homogeneous spinor flow is then a dynamical system on the finite-dimensional subspace $\mc{N}^G \subset \mc{N}$ of $G$-invariant sections of the universal spinor bundle. The spinorial energy functional and the $L^2$-metric on $\mc{N}^G$ can be explicitly computed. In some situations the space $\mc{N}^G$ and the spinorial energy functional is sufficiently simple to describe the behaviour of the spinor flow. In this paper, the calculations are done for unimodular three dimensional Lie groups, almost abelian Lie groups and the flag manifold in dimension six.
\section{The spinor flow in general}
\subsection{The universal spinor bundle and the spinorial energy}
Let $M$ be a spin manifold. By this we mean a smooth oriented $n$-dimensional manifold such that the principal $\GL_+(n)$ bundle of oriented frames $P$ admits a double cover $\pi:\tilde P\rightarrow P$, where $\tilde P$ is a  $\widetilde{\GL}_+(n)$ principal bundle, such that the group action commutes with the covering map, i.e.\@ the following diagram commutes
\begin{displaymath}
  \xymatrix{ \tilde P \times \widetilde{\GL}_+(n) \ar[d] \ar[r] & \tilde P \ar[d]^{\pi} \ar[dr] & \\
    P \times \GL_+(n) \ar[r] & P \ar[r] & M.} 
\end{displaymath}
The principal bundle $\tilde P$ is called a {\em topological spin structure} on $M$.

A Riemannian metric $g$ on $M$ defines a reduction of $P$ to the structure group $\SO(n) \subset \GL_+(n)$. This reduction is given by the bundle $P_g$ of oriented orthonormal frames of $(M,g)$. The preimage $\tilde P_g = \pi^{-1}(P_g) \subset \tilde P$ defines a $\Spin(n)$ reduction of $\tilde P$. The bundle $\tilde P_g$ is a {\em spin structure} on $(M,g)$. Thus, for a spin manifold all spin structures arise as subbundles of the topological spin structure $\tilde P$. This observation can be used to fit all spinor bundles into one fiber bundle, called the {\em universal spinor bundle}.

Recall that the complex spinor bundle of a spin manifold $(M,g)$ is given by
$$\Sigma_g M = \tilde{P}_g \times_{\rho_n} \Sigma_n,$$
where $\Sigma_n = \C^{2^{[n/2]}}$ and $\rho_n : \Spin(n) \to \End(\Sigma_n)$ is the standard spin representation. Since
\[
\widetilde{\GL}_+(n) \to \widetilde{\GL}_+(n) / \Spin(n)  \cong \odot^2_+ (\R^n)^* 
\]
is a $\Spin(n)$-principal bundle, we can form the associated vector bundle
\[
F_n = \widetilde{\GL}_+(n) \times_{\rho_n} \Sigma_n \to \odot^2_+ (\R^n)^*,
\]
which carries an action from the left of $\widetilde{\GL}_+(n)$. The {\em universal spinor bundle} is by definition the fiber bundle
\[
\Sigma M = \tilde P \times_{\widetilde{\GL}_+(n)} F_n
\]
with typical fiber the total space of the vector bundle $F_n$. It thus has a double fibration structure
\[
\xymatrix@R=.5cm{ \Sigma M \ar[dd] \ar[rd]& \\ & \odot^2_+ T^*\! M \ar[ld] \\ M &}
\]
where $\odot^2_+ T^*\! M\cong \tilde P/\Spin(n)$ is the bundle of positive definite symmetric bilinear forms on $TM$ and $\Sigma M \to \odot^2_+ T^*\! M$ is a vector bundle with typical fiber $\Sigma_n$. 

A section $\Phi$ of $\Sigma M$ may be identified with a $\widetilde{\GL}_+(n)$-equivariant map $\Phi : \tilde P \to F_n$ so that the 
projection $F_n \to \odot^2_+ (\R^n)^*$ determines a section $g_{\Phi}$ of $\odot^2_+ T^*\!M$, i.e.\ a Riemannian metric on $M$. The 
metric $g_{\Phi}$ defines a geometric spin structure $\tilde P_{g_{\Phi}}\subset \tilde P$, which is explicitly given by
$\tilde P_{g_{\Phi}}=\Phi^{-1}(\Spin(n) \times_{\rho_n} \Sigma_n)$. Thus the restriction of
$\Phi$ to $\tilde P_{g_{\Phi}}$ naturally defines a $\Spin(n)$-equivariant map $\varphi_{\Phi}:\tilde P_{g_{\Phi}}\rightarrow \Sigma_n$,
i.e. a section $\varphi_{\Phi}$ of the spinor bundle $\Sigma_{g_{\Phi}} M$. 
So to any universal spinor field $\Phi$, we may associate the pair 
\[
(g_{\Phi}, \varphi_{\Phi}) \in \Gamma(\odot^2_+ T^*\!M) \times \Gamma(\Sigma_{g_\Phi}M). 
\]
Conversely, a metric $g$ and a section $\varphi$ of $\Sigma_g M$ define in a unique way a section of $\Sigma M$. It is this property that leads us to call $\Sigma M$ the {\em universal spinor bundle} of the spin manifold $M$.
 The space of sections of $\Sigma M$ will also be denoted by $\mc{F}$.
The space of unit length sections of the universal spinor bundle
$$\mathcal{N} = \{\Phi = (g, \varphi) \in \Gamma(\Sigma M) : |\varphi| \equiv 1\} \subset \mc F$$
is the domain of the spinorial energy functional. To define this functional, recall that the Levi--Civita connection on $TM$ induces a connection $\nabla^g$ on $\Sigma_g M$.

The {\em spinorial energy functional} is given by
$$\mc{E} : \mc{N} \to \R$$
$$\mc{E}(g, \varphi) = \frac{1}{2} \int_M |\nabla^g \varphi|^2 \vol_g.$$
This energy functional has several important symmetries. We restrict here to the invariance under so-called
\emph{spin diffeomorphisms} and refer for the other symmetries to \cite{aww1}.

A {\em spin diffeomorphism} is a diffeomorphism $f:M\to M$ for which the induced map $df : P \to P$ lifts to the spin structure $\tilde{P}$. We denote by $\Diff_S(M)$ the group of such spin diffeomorphisms of $M$. Then the lifts of elements of $\Diff_S(M)$ themselves form a group $\widehat{\Diff}_S(M)$ and there is an exact sequence
$$0 \to \bZ_2 \to \widehat{\Diff}_S(M) \to \Diff_S(M) \to 0.$$
The group $\widehat{\Diff}_S(M)$ acts on the universal spinor bundle $\Sigma M$. Let $F: \tilde P \to \tilde P$ be an element of $\widehat{\Diff}_S(M)$.
Considering elements of $\Gamma(\Sigma M)$ as $\widetilde{\GL}_+(n)$-equivariant maps $\tilde P \to F_n$, the map $F$ acts on $\Gamma(\Sigma M)$ by precomposition, i.e.\@ $F^* \Phi = \Phi \circ F$ for $\Phi\in \Gamma(\Sigma M)$. One easily checks that $\cE(F^*\Phi)=\cE(\Phi)$
for any universal spinor field $\Phi$.
\subsection{The spinor flow}
The {\em spinor flow} is the negative gradient flow of the spinorial energy functional $\mc{E}$ with respect to the natural $L^2$-metric on $\mc{N}$, which we will describe momentarily. Thereto, look first at the principal $\mathrm{SO}(n)$-bundle $\GL_+(n)\rightarrow \GL_+(n)/\mathrm{SO}(n)=\odot^2_+ (\R^n)^*$ and
note that the decomposition of an $n\times n$-matrix into its symmetric and anti-symmetric part naturally induces a 
splitting of $T\GL_+(n)$ into a horizontal and vertical distribution, see the text before Lemma \ref{H-action_parallel} for an explicit description. So we get a natural  $\mathrm{SO}(n)$-connection on this principal bundle and so also a natural principal $\mathrm{Spin}(n)$-connection on the $\Spin(n)$-bundle $\widetilde{\GL}_+(n) \to \widetilde{\GL}_+(n) / \Spin(n)=\odot^2_+ (\R^n)^*$. Both of these connections are known as {\em Bourguignon--Gauduchon connection} as they were first introduced by these two authors in their joint paper \cite{bg}.
Using the Bourguignon--Gauduchon connection, one may define a natural horizontal distribution on the vector bundle $\Sigma M\rightarrow \odot^2_+ T^* M$
and so obtains an isomorphism
$$T_{g_x, \varphi_x} \Sigma M_x \cong \odot^2 T_x^* M \oplus (\Sigma_g M)_x$$
for all $x \in M$.
This isomorphism yields an identification
$$T_{(g, \varphi)} \mc{N} = \Gamma(\odot^2 T^* M \oplus \Sigma_g M^{\perp \varphi}),$$
i.e. we may split the tangent space $T_{(g, \varphi)} \mc{N}$ into metric and spinorial directions.

On $\Gamma(\odot^2 T^* M)$ there is a natural $L^2$-metric $((\cdot\,,\cdot))_g$, which arises from integrating the pointwise scalar product $(\cdot\,,\cdot)_g$
on $\odot^2 T^* M$ induced by $g$. Similary, one obtains a natural $L^2$-metric $\langle \langle \cdot\,,\cdot\rangle\rangle_g$ on $\Sigma_g M$ by integrating the pointwise scalar product $\langle \cdot\,,\cdot\rangle$ on $\Sigma_g M$, the latter being the one induced by the natural $\mathrm{Spin}$-invariant scalar product on $\Sigma_n$. The gradient of the functional $\mc{E}$ may now be defined With respect to that $L^2$-metric
on $T_{(g, \varphi)} \mc{N}$. The negative gradient of $\mc{E}$ will be denoted by $Q$, i.e.\@
\begin{equation}\label{eq:defQ}
Q(g,\varphi) = - \grad \mc{E}(g, \varphi).
\end{equation}
The \emph{spinor flow} is then defined by the equation
\begin{equation}\label{eq:spinorflow}
\partial_t (g_t, \varphi_t) = Q(g_t, \varphi_t).
\end{equation}
The negative gradient $Q$ can be split into a metric part $Q_1(g, \varphi) \in \Gamma(\odot^2 T^* M)$ and a spinorial part $Q_2(g, \varphi) \in \Gamma(\Sigma_g M^{\perp \varphi})$. In \cite{aww2}, the following formulas for the components of $Q$ were found:
\begin{align*}
  Q_1(g, \varphi) & = -\frac{1}{4} |\nabla^g \varphi| g - \divgc_g T_{g, \varphi} + \frac{1}{2} \langle \nabla^g \varphi \otimes \nabla^g \varphi \rangle,\\
  Q_2(g, \varphi) & = - \nabla^{g*} \nabla^g \varphi + |\nabla^g \varphi|^2 \varphi,
\end{align*}
where
$$T_{g,\varphi}(X,Y,Z) = \frac{1}{2} \left( \langle X \cdot Y \cdot \varphi, \nabla^g_Z \varphi \rangle + \langle X \cdot Z \cdot \varphi, \nabla^g_Y \varphi \rangle \right)$$
and
$$ \langle \nabla^g \varphi \otimes \nabla^g \varphi \rangle(X,Y) = \langle \nabla^g_X \varphi, \nabla^g_Y \varphi \rangle.$$
To understand the long-time behaviour of the spinor flow, it is imperative to understand finite-time singularities of that flow. Self-similar solutions of the spinor flow are possible singularities.
A particular class of self-similar solutions are those which evolve only by scaling the metric, or equivalently critical points of the volume-normalized spinor flow. To obtain such critical points one restricts the spinorial energy functional
$\cE$ to $\mathcal{N}_1 := \{\Phi = (g, \varphi) \in \mathcal{N} : \vol^g(M)=1\}$.
Using that $T_{(g,\varphi)} \mathcal{N}_1=\left\{(\dot{g},\dot{\varphi})\in T_{(g,\varphi)}\mathcal{N}: ((\dot{g},g))_g=0 \right\}$, one 
computes that the negative gradient $\tilde{Q}(g,\varphi)=(\tilde{Q}_1(g,\varphi),\tilde{Q}_2(g,\varphi))$ of $\cE|_{\mathcal{N}_1}$ 
fulfills
\begin{equation}\label{eq:tildeQ}
\tilde{Q}_1(g,\varphi)=Q_1(g,\varphi)+\frac{n-2}{2n}\cE(g,\varphi)\, g,\qquad \tilde{Q}_2(g,\varphi)=Q_2(g,\varphi)
\end{equation}
and the \emph{volume-normalized spinor flow} is given by
\begin{equation}\label{eq:volumenormspinflow}
\partial_t (g_t, \varphi_t) = \tilde{Q}(g_t, \varphi_t)
\end{equation}
Since $\cE(F^*\Phi)=\cE(\Phi)$ for any $\Phi\in \mc{M}$ and any $F\in \widehat{\Diff}_S(M)$ one obtains directly:
\begin{proposition}\label{pro:spinorflowspindiffeos}
Let $\Phi\in \mc{N}$, $\tilde\Phi\in \mc{N}_1$ and $F\in \widehat{\Diff}_S(M)$. Moreover, let
$(\Phi_t)_{t\in I}$ be the solution of the spinor flow with initial value $\Phi$
and $(\tilde\Phi_t)_{t\in I}$ be the solution of the volume-normalized spinor flow with initial value $\tilde\Phi$.

Then $Q(F^*\Phi)=F^*Q(\Phi)$ and $\tilde{Q}(F^*\tilde\Phi)=F^*\tilde{Q}(\tilde\Phi)$ and, consequently,
$(F^*\Phi_t)_{t\in I}$ is the solution of the spinor flow with initial value $F^*\Phi$ and
$(F^*\tilde\Phi_t)_{t\in I}$ is the solution of the volume-normalized spinor flow with initial value $F^*\tilde{\Phi}$.
\end{proposition}
\section{Universal spinor fields on homogeneous manifolds}
For the purposes of this article it is clearly necessary to understand the universal spinor bundle and its invariant sections on a homogeneous manifold.

A manifold $M$ is called \emph{homogeneous} if a Lie group $G$ acts transitively on it. Fix now $x \in M$ for the rest of this article. Then $g\cdot x\mapsto gH$ is a diffeomorphism from $M$ to the quotient manifold $G/H$, where $H$ is the stabilizer subgroup of the point $x \in M$. Hence, we may identify $M$ with $G/H$ and so $T_x M$ with $T_{eH} G/H$. Throughout this article the natural assumptions that $G$ is connected and simply connected, that $H$ is compact and its action of $H$ is {\em reductive} are made. The action of $H$ is by definition reductive, if the Lie algebra $\mf{g}$ splits as $\mf{h} \oplus \mf{p}$, where $\mf{h}$ is the Lie algebra of $H$ and $\mf{p}$ is invariant under the adjoint action of $H$. Denoting then the canonical projection from $G$ to $G / H$ by $\pi$, the restriction of $d \pi_e: \mf{g} \to T_{eH} G/H\cong T_x M$ to $\mf{p}$ is an isomorphism and we will from now on identify $T_x M$ with $\mf{p}$.

Assume now that $M$ is an oriented manifold and that $G$ acts in an orientation preserving manner. Consider the {\em isotropy representation}
$$\alpha: H \to \GL_+(\mf p),\qquad h \mapsto (dl_h)_e|_{\mf p}$$
and choose a $G$-invariant background metric $\bar{g}$. Note that such a metric is given
by an $H$-invariant scalar product on $\mf{p}$ and that such a metric exists on $M$ if and only if $H$ is compact. Since we assumed $H$ to be compact, this is the case. If $M$ is compact, we additionally assume that
$\vol^{\bar{g}}(M)=\int_M \vol^{\bar{g}}=1$.
Fix an oriented orthonormal basis $(X_1, \ldots, X_n)$ of $\mf{p} = T_x M$, identifying $\mf{p}$ with $\R^n$. With respect to this identification the isotropy representation $\alpha$ becomes a homomorphism into $\mathrm{SO}(n)\subseteq\GL_+(n)$.

The tangent bundle of $M$ can be constructed from the group $G$ and the isotropy representation $\alpha: H \to \GL_+(n)$ as the associated bundle $G \times_{\alpha} \R^n$. Notice that the isomorphism
$$TM \to G \times_{\alpha} \R^n$$
depends on the choice of the basis $(X_1,\ldots,X_n)$ of $\mf{p}\cong T_x M$. The bundle $P$ of oriented frames of $M$ is isomorphic to $G \times_{\alpha} \GL_+(n)$. Assuming there is a lift
$$\tilde{\alpha} : H \to \widetilde{\GL}_+(n)$$
of the isotropy representation, the associated bundle
$$\tilde P = G \times_{\tilde \alpha} \widetilde{\GL}_+(n)$$
forms a topological spin structure. Thus, a topological spin structure compatible with the homogeneous structure is given by a lift of the isotropy representation, leading to the following definition.
\begin{definition}
  An {\em equivariant spin structure} on a homogeneous manifold $G/H$ is a lift $\tilde \alpha: H \to \widetilde{\GL}_+(n)$ of the isotropy representation $\alpha: H \to \GL_+(n)$.
\end{definition}
Given such an equivariant spin structure, the universal spinor bundle equals
$$\Sigma M = G \times_{\tilde \alpha} F_n,$$
where $\widetilde{\GL}_+(n)$ acts on $F_n$ via
$$\tilde A \cdot [\tilde B, \varphi] = [\tilde A \tilde B, \varphi].$$
A {\em universal spinor field} is a section of $\Sigma M$ and may hence be identified with an $H$-equivariant smooth map
$$\Phi: G \to F_n.$$

Now note that as the group of diffeomorphisms isotopic to the identity $\Diff_0(M)$ is a subgroup of $\Diff_S(M)$, the group $G$ acts by 
spin diffeomorphisms. Hence, we may define:
\begin{definition}
  A {\em $G$-invariant universal spinor field} is a section $\Phi$ of $\Sigma M$, which is invariant under the action of $G$, i.e.\@ which satisfies $F^* \Phi = \Phi$ for any $F \in \widehat{\Diff}_S(M)$ that arises as the lift of a $g \in G \subset \Diff_S(M)$.
\end{definition}
With respect to the identification $\Sigma M = G \times_{\tilde \alpha} F_n$, the action of $G$ on $\mc{F}$ is given by left-translation 
in the argument of $F\in \mc{F}$, $F:G\rightarrow F_n$, and so a $G$-invariant universal spinor field $\Phi$ corresponds to a constant $H$-equivariant map $G \to F_n$ or in other words an $H$-invariant element in $F_n$. Similarly, a $G$-invariant metric $g$ on $M$ corresponds to an $H$-invariant element in $\odot^2_+\mf{p}^* \cong \odot^2_+ (\R^n)^*$. Let $\mc{F}^G = F_n^H$ denote the space of $G$-invariant universal spinor fields, $\mc{N}^G\subseteq \mc{F}^G$ the subspace of $G$-invariant universal spinor fields of unit length
and $\mc{M}^G = (\odot^2_+ (\R^n)^*)^H$ the space of $G$-invariant metrics on $M$.

The next lemma gives us a more explicit description of the space $\mc{F}^G$:
\begin{lemma}\label{H-inv}
An element $\Phi=[\tilde A,\varphi] \in F_n$ is $H$-invariant if and only if
\medskip
\begin{enumerate}[(i)]
\item
$\tilde{A}^{-1}\tilde{\alpha}(h) \tilde{A} \in \Spin(n)$ for all $h\in H$, and
\smallskip
\item
$\varphi$ is fixed by the $H$-representation $H\ni h\mapsto \rho_n(\tilde{A}^{-1}\tilde{\alpha}(h) \tilde{A})\in \GL(\Sigma_n)$.
\end{enumerate}
\end{lemma}

\begin{proof}
The action of $H$ on $F_n$ is given by
\[
h \cdot [\tilde A, \varphi] = [\tilde{\alpha}(h)\tilde A, \varphi].
\]	
Furthermore, $[\tilde{\alpha}(h) \tilde A, \varphi] = [\tilde A, \varphi] \in F_n$ if and only if there exists $\tilde b \in \Spin(n)$ such that $\tilde{\alpha}(h) \tilde A = \tilde{A}\tilde b$ and $\varphi = \rho_n(\tilde b^{-1}) \varphi$. Altogether this yields the claim.
\end{proof}
For the rest of the paper we assume that $\mc{F}^G\neq \emptyset$ and denote by $A\in \GL_+(n)$ the image of 
$\tilde{A}$ under the covering map $\widetilde{\GL}_+(n)\rightarrow \GL_+(n)$. Note that 
\[
\pi:\mc{F}^G\rightarrow \mc{M}^G, \quad
 [\tilde{A},\varphi]\mapsto g:=g^A:=\bar{g}(A^{-1}\cdot,A^{-1}\cdot)
\] 
is the projection of a vector bundle, condition (i)
in Lemma \ref{H-inv} ensuring that the projection lands in $\mc{M}^G$.
In particular, $\mc{M}^G\neq \emptyset$ as well. Moreover, we explicitly get on $\bR^n$ that
$g^A(v,w)=v^T A^{-T} A^{-1} w$ for all $v,w\in\bR^n$ and so may further identify $g^A\in \mc{M}^G\cong (\odot_+^2 (\R^n)^*)^H$ with
$A^{-T}A^{-1}\in \Sym_+(n,\bR)^H$, the latter being the space of $H$-invariant symmetric positive definite real $n\times n$-matrices.

Coming back to Lemma \ref{H-inv}, we see that the invariance condition for the spinorial part $\varphi$
of the $G$-invariant universal spinor $\Phi=[\tilde{A},\varphi]$ depends on the chosen $\tilde{A}$ and so may not be formulated
consistently for all universal spinors. However, we are only interested in getting such a consistent invariance condition
for a solution of the spinor flow $t\mapsto \Phi_t=[\tilde{A}_t,\varphi_t] \in \mc{N}^G$ in our homogeneous setting. This may be achieved
as follows:

Consider an arbitrary path $g_t$ of $G$-invariant metrics on $M$ and let $A_t$ in $\GL_+(n)$ and $\tilde{A}_t$ in $\widetilde{\GL}_+(n)$ 
be the horizontal lifts of this path with respect to the Bour\-guignon-Gauduchon connections. Moreover, recall that the
horizontal distribution $\cH$ defining the Bourguignon-Gauduchon connection on the principal $\SO(n)$-bundle
$\GL_+(n) \to \GL_+(n) / \SO(n) \cong \odot^2_+ \R^n$ and the vertical distribution $\cV$ are explicitly given by
\begin{align*}
\cH_A &= \{M \in \R^{n \times n} = T_A \GL_+(n) :  (A^{-1} M)^T = A^{-1} M\},\\
\cV_A &= \{M \in \R^{n \times n} = T_A \GL_+(n): (A^{-1} M)^T = -A^{-1} M\}.
\end{align*}
for $A\in \GL_+(n)$. Using these explicit descriptions of $\cH$ and $\cV$, one obtains:
\begin{lemma}\label{H-action_parallel}
	If $A_t$ and $\tilde{A}_t$ are horizontal lifts of a path of $G$-invariant metrics $g_t$, then the representations 
	\[
	A^{-1}_t \alpha A_t: H \to \SO(n), \, h \mapsto A^{-1}_t \alpha(h) A_t
	\]
	and
	\[
	\tilde A^{-1}_t \tilde\alpha \tilde A_t: H \to \Spin(n), \, h \mapsto \tilde A^{-1}_t \tilde\alpha(h) \tilde A_t
	\]
	do not depend on $t$.
\end{lemma}
\begin{proof}
Fix $h \in H$. Since $A^{-1}_t \alpha(h) A_t \in \SO(n)$ for all $t$, clearly 
\[
\frac{d}{dt} \left(A_t^{-1}\alpha(h)A_t\right) \in \cV_{A_t^{-1} \alpha(h) A_t}.
\]
On the other hand, since $A_t$ is a horizontal curve, $\dot{A}_t:=\frac{d}{dt}A_t\in \cH_{A_t}$ for all $t$ and hence
\begin{equation*}
\frac{d}{dt} \left(A_t^{-1}\alpha(h)A_t\right)=-A_t^{-1} \dot{A}_t A_t^{-1} \alpha(h) A_t+A_t^{-1} \alpha(h) \dot{A}_t \in \cH_{A_t^{-1} \alpha(h) A_t}
\end{equation*}
Here we have used that the horizontal distribution $\cH$ is left-invariant by $\GL_+(n)$ and right-invariant by $\SO(n)$. However, since vertical and horizontal spaces are complementary, this implies that $\frac{d}{dt} \left(A_t^{-1}\alpha(h)A_t\right)=0$, i.e.\ $t\mapsto A_t^{-1}\alpha(h)A_t$ is constant.
Since $\tilde{A}_t^{-1}\tilde{\alpha}(h)\tilde{A}_t$ is mapped to $A_t^{-1}\alpha(h)A_t$ by the covering $\widetilde{\GL}_+(n)\rightarrow \GL_+(n)$,
the map $h \mapsto \tilde A^{-1}_t \tilde\alpha(h) \tilde A_t$ has to be constant as well.
\end{proof}

As a consequence, the vector bundle $\pi: \mc{F}^G \to \mc{M}^G$ may be trivialized along a path of $G$-invariant metrics by parallel transport. 

\begin{lemma}\label{le:trivialization}
Let $g\in \mc{M}^G  \cong \Sym_+(n,\bR)^H$ 
and $\tilde{A}\in \widetilde{\GL}_+(n)$ be such that $g=A^{-T} A^{-1}$.
If $\gamma: I =[0,T]  \to \mc{M}^G, t \mapsto g_t$ is a smooth path with $g_0 = g$, then
	\[
	\gamma^* \mc{F}^g \cong I \times \Sigma_n^H,
	\]
	where $H$ acts on $\Sigma_n$ by $\rho_n \circ \tilde A^{-1} \tilde \alpha \tilde A$.
\end{lemma}

\begin{proof}
Let $\tilde A_t$ be the horizontal lift of $g_t$ to $\widetilde{\GL}_+(n)$ with $\tilde{A}_0=\tilde A$, then $\tilde A^{-1}_t \tilde\alpha \tilde A_t = \tilde A^{-1} \tilde \alpha \tilde A$ for all $t \in I$ according to Lemma \ref{H-action_parallel}. Then
\[
I \times \Sigma_n^ H \to \gamma^* \mc{F}^g, \quad (t, \varphi) \to [\tilde A_t, \varphi]
\]
provides a trivialization.
\end{proof}

If $\Phi_t$ is a family of $G$-invariant universal spinor fields covering the family of $G$-invariant metrics $g_t$ (with $g_0=g$), 
we may therefore identify it with the family
\[
(g_t, \varphi_t) \in \mc{M}^G \times \Sigma_n^H\cong \Sym_+(n,\bR)^H\times \Sigma_n^H
\]
obtained by writing $\Phi_t=[ \tilde A_t, \varphi_t]$ for $\tilde A_t$ the horizontal lift of $g_t$. This will apply in particular to a solution of the spinor flow 
in our homogeneous setting.

\begin{remark*}
For practical purposes, we will compute $\mc{F}^G$ by linearization of the group action, i.e.\@ we will compute $\Sym_+(n,\bR)^{\mf{h}}$
instead of $\Sym_+(n,\bR)^H\cong \mc{M}^G$. Similarly, we determine $\Sigma_n^{\mf{h}}$ instead of $\Sigma_n^H$.
\end{remark*}

\section{The spinorial flow on homogeneous manifolds}

\label{SpEn}
\subsection{Spinorial energy in the homogeneous setting}
The considerations of the former section enable us to explicitly calculate the spinorial energy of an invariant section $\Phi$ of the universal spinor bundle $\Sigma M$ of a \emph{compact} $G$-homogeneous space $M=G/H$ with properties as before. We use all the notations and identifications from the last section and first
of all choose a representative $(\tilde{A},\varphi)\in \widetilde{\GL}_+(n)\times \Sigma_n$ of $\Phi$, i.e. $\Phi=[\tilde{A},\varphi]$, and
let $A$ be the image of $\tilde A$ under the map $\widetilde{\GL}_+(n)$ to $\GL_+(n)$. We will express
the spinorial energy $\cE(\Phi)$ of $\Phi$ in terms of this representative and note that the formula is of course independent of the 
chosen representative although one may not spot this directly when looking at the formula. 

We first will derive a formula for the spin connection adapted to our setting, namely the one already given in \cite{bae}. For 
completeness and to help the reader in understanding all the above identifications, we recapitulate the computations
done in \cite{bae} to derive that formula. Let us start with the Levi-Civita connection for vector fields.
As we have a $G$-homogeneous space, there are distinguished vector fields given by the fundamental vector fields $\overline{X} \in \Gamma(TM)$ of 
elements $X\in \mf{g}$. These fundamental vector 
fields are Killing for the induced $G$-invariant Riemannian metric $g=g^A$ and fulfill $\overline{X}(x)=X$ if $X\in \mf{p}=T_x M$ and 
$\overline{X}(x)=0$ if $X\in \mf{h}$. The following calculations will all be made at the point $x \in M$. In this setting the Levi--Civita connection satisfies the identity
\begin{equation*}
g(\nabla_{\overline{X}}^g \overline{Y}, \overline{Z})(x)
=-\tfrac{1}{2}\left(g(\overline{[X,Y]_{\mf{p}}},\overline{Z})+g(\overline{[X,Z]_{\mf{p}}},\overline{Y})+g(\overline{[Y,Z]_{\mf{p}}},\overline{X})\right)(x)
\end{equation*}
by \cite[7.28 Proposition]{besse} for all $X,Y,Z\in \mf{g}$, where the index $\mf{p}$ denotes the projection $\mf{g}\rightarrow \mf{p}$ along $\mf{h}$.

Now note that we have we have fixed an oriented basis $(X_1,\ldots,X_n)$ of $\mf{p}$, which constitutes an orthonormal basis for
a $G$-invariant background metric $\bar{g}$. As we have chosen a specific representative $(\tilde{A},\varphi)\in \widetilde{\GL}_+(n)\times \Sigma_n$, we also 
have a natural orthonormal basis $(Y_1,\ldots,Y_n):=(A X_1,\ldots,A X_n)$ of $T_x M=\mf{p}$ for the $G$-invariant Riemannian metric 
$g=g^A=\bar{g}(A^{-1}\cdot,A^{-1}\cdot)$. We set
\begin{equation*}
\begin{split}
c_{ijk}(A):=&-g([Y_k,Y_i]_{\mf{p}},Y_j)-g([Y_k,Y_j]_{\mf{p}},Y_i)-g([Y_i,Y_j]_{\mf{p}},Y_k)\\
=&\, \bar{g}(A^{-1}[A X_i,A X_k]_{\mf{p}},X_j)+\bar{g}(A^{-1}[A X_j,A X_k]_{\mf{p}},X_i)\\
&-\bar{g}(A^{-1}[A X_i,A X_j]_{\mf{p}},X_k)
\end{split}
\end{equation*}
for all $i,j,k=1,\ldots,n$ and observe that 
\begin{equation*}
g(\nabla_{Y_k} \overline{Y}_i,Y_j)=\tfrac{1}{2}c_{ijk}(A).
\end{equation*}
for all $i,j,k=1,\ldots,n$.

Choose now some $\tilde{b}\in \tilde{P}_g$ covering $b:=(Y_1,\ldots,Y_n)\in P_g$. Then 
\mbox{$G\times_{\tilde{A}^{-1}\tilde{\alpha} \tilde{A}} \Spin(n)$} is isomorphic to $\tilde{P}_g$ via the map
$[\tilde{g},\tilde{B}]\mapsto \tilde{g}(\tilde{b})\cdot \tilde{B}$ for $\tilde{g}:\tilde{P}_g\rightarrow \tilde{P}_g$ covering $d\tilde{g}_x:P_g\rightarrow P_g$.
Thus, we also have
\begin{equation*}
\Sigma_g M=\tilde{P}_g\times_{\rho_n} \Sigma_n\cong (G\times_{\tilde{A}^{-1}\tilde\alpha \tilde{A}} \Spin(n))\times_{\rho_n} \Sigma_n \cong G\times_{\rho_n(\tilde{A}^{-1}\tilde\alpha \tilde{A})} \Sigma_n,
\end{equation*}
where the latter isomorphism is given by $[[\tilde{g},\tilde{B}],\tilde{\varphi}]\mapsto [\tilde{g},\rho_n(\tilde{B})\tilde{\varphi}]$ with inverse
$[\tilde{g},\tilde{\varphi}]\mapsto [[\tilde{g},1_{\Spin(n)}],\tilde{\varphi}]$. In our case, when identifying $\Sigma_g M\cong G\times_{\rho_n(\tilde{A}^{-1}\tilde\alpha \tilde{A})} \Sigma_n$, the $G$-invariant spinor field induced by the $G$-invariant universal spinor $\Phi=[\tilde{A},\varphi]$ is given by $M \ni \tilde{g}\cdot x\mapsto [\tilde{g},\varphi]\in \Sigma_g M$. Hence, by an abuse of notation, we also denote the $G$-invariant spinor field itself by $\varphi$.

At this point, recall that the Levi-Civita connection of $g$ corresponds to a unique principal $\mathrm{SO}(n)$-connection on $P_g$, 
which then uniquely lifts to a $\Spin(n)$-connection on $\tilde{P}_g$ inducing the \emph{spin connection} $\nabla^g$ on the associated 
bundle $\Sigma_g M$. By identifying sections of $\Sigma_g M$ with $\Spin(n)$-equivariant maps
$\sigma:\tilde{P}_g\rightarrow \Sigma_n$, the spin connection $\nabla^g$ is explicitly given by $\nabla_X \sigma=X^*(\sigma)$,
where $X^*\in \mathfrak{X}(\tilde{P}_g)$ is the unique horizontal lift of $X\in \mathfrak{X}(M)$.
Now let $\sigma:\tilde{P}_g\rightarrow \Sigma_n$ be the $\Spin(n)$-equivariant map associated to $\varphi$ and let $\tilde{B}\in \Spin(n)$.
and $\tilde{g}\in G$. Then $[\tilde{g},\tilde{B}]\in G\times_{\tilde{A}^{-1}\tilde{\alpha} \tilde{A}} \Spin(n)\cong \tilde{P}_g$
is in the fibre of $\tilde{P}_g$ over $\tilde{g}\cdot x$ and with the above identifications we obtain
\begin{equation*}
[\tilde{g},\varphi]=\varphi(\tilde{g}\cdot x)=[[\tilde{g},\tilde{B}],\sigma([\tilde{g},\tilde{B}])]=[\tilde{g},\rho_n(\tilde{B})\sigma([\tilde{g},\tilde{B}])],
\end{equation*}
i.e. $\sigma([\tilde{g},\tilde{B}])=\rho_n(\tilde{B}^{-1})(\varphi)$.

Now let $Y\in \mf{p}$. Then $c(t):=\exp(tY)\cdot x$ fulfills $\dot{c}(0)=\overline{Y}(x)=Y$ and so there exists a (unique) curve $t\mapsto \tilde{B}(t)\in\Spin(n)$
with $\tilde{B}(0)=1_{\Spin(n)}$ such that $\tilde{c}(t):=[\exp(tY),\tilde{B}(t)]\in \tilde{P}_g$ is horizontal and fulfills 
$\dot{\tilde{c}}(0)=\overline{Y}^*(x)$. Thus, we obtain
\begin{equation*}
\begin{split}
\nabla^g_Y \sigma &=(\nabla^g_{\overline{Y}} \sigma) (x)=\left.\frac{d}{dt}\right|_{t=0} \sigma([\exp(tY),\tilde{B}(t)])
=\left.\frac{d}{dt}\right|_{t=0} \rho_n(\tilde{B}(t)^{-1})(\varphi)\\
&=-(\rho_n)_{*}(\dot{\tilde{B}}(0))(\varphi).
\end{split}
\end{equation*}
So we are left with computing $\dot{\tilde{B}}(0)$. Since
\begin{equation*}
t\mapsto [\exp(tY),B(t)]=d(\exp(tY))_x(b)\cdot B(t)
\end{equation*}
is a horizontal curve in $P_g$, all the vector fields along $c$ contained in this horizontal curve are parallel.
Hence, we get
\begin{equation*}
0=\left.\frac{\nabla^g}{dt}\right|_{t=0} \Bigl(d(\exp(tY))_x(b)\cdot B(t)\Bigr)
 =\left.\frac{\nabla^g}{dt}\right|_{t=0} d(\exp(tY))_x(b) + db_x(\dot{B}(0)),
\end{equation*}
where $b$ also denotes the map $\mathrm{SO}(n)\ni B\mapsto b\cdot B\in (P_g)_x$.
Now note that
\begin{equation*}
d(\exp(tY))_x(b)=d(\exp(tY))_x(Y_1),\ldots,\exp(tY))_x(Y_n))
\end{equation*}
and
\begin{equation*}
\begin{split}
\left.\frac{\nabla^g}{dt}\right|_{t=0} d(\exp(tY))_x(Y_k)
&=\left.\frac{\nabla^g}{dt}\right|_{t=0}\left.\frac{\partial}{\partial s}\right|_{s=0} \exp(tY)\cdot(\exp(sY_k)\cdot x)\\
&=\left.\frac{\nabla^g}{ds}\right|_{s=0}\left.\frac{\partial}{\partial t}\right|_{t=0} \exp(tY)\cdot(\exp(sY_k)\cdot x)\\
&=\left.\frac{\nabla^g}{ds}\right|_{s=0} \overline{Y}(\exp(sY_k)\cdot x)=\nabla^g_{\overline{Y}_k(x)} \overline{Y}\\
&=\sum_{j=1}^n g\left(\nabla^g_{Y_k} \overline{Y},Y_j\right) Y_j.
\end{split}
\end{equation*}
for all $k=1,\ldots,n$. Moreover, 
\begin{equation*}
db_x(\dot{B}(0))=\tfrac{d}{dt} \Bigl(b\cdot \exp(t \dot{B}(0))\Bigr)=b\cdot \dot{B}(0):=\left(\sum_{j=1}^n \dot{B}(0)_{j1} Y_j,\ldots,\sum_{j=1}^n \dot{B}(0)_{jn} Y_j\right),
\end{equation*}
Hence, $\dot{B}(0)_{jk}=-g(\nabla^g_{Y_k} \overline{Y},Y_j)\in \mathfrak{so}(n)$
and so
\begin{equation*}
\dot{\tilde{B}}(0)_{jk}=-\frac{1}{2}\sum_{j,k=1}^n g(\nabla^g_{Y_k} \overline{Y},Y_j)\, E_j\cdot E_k\in \mathfrak{spin}(n).
\end{equation*}
This implies first of all
\begin{equation*}
\begin{split}
\nabla^g_{Y_i} \sigma&=\frac{1}{2}\sum_{j,k=1}^n g(\nabla_{Y_k} \overline{Y}_i,Y_j)\, E_j\cdot E_k\cdot \varphi
                   =\frac{1}{4}\sum_{1\leq j<k\leq n} g(\nabla_{Y_k} \overline{Y}_i,Y_j)\, E_j\cdot E_k\cdot \varphi\\
									&=\frac{1}{4}\sum_{1\leq j<k\leq n} c_{ijk}(A) E_j\cdot E_k\cdot \varphi
\end{split}
\end{equation*}
for any $i=1,\ldots,n$ and then
\begin{equation*}
|\nabla^g \varphi|^2(x) =\sum_{i=1}^n |\nabla^g_{Y_i} \varphi|^2
=\frac{1}{16} \sum_i \Bigl|\sum_{j<k} c_{ijk}(A) E_j \cdot E_k \cdot \varphi\Bigr|^2.
\end{equation*}
Since $\nabla^g \varphi$ is $G$-invariant, the function $|\nabla^g \varphi|^2$ is constant on $M$. Thus,
\begin{equation*}
\int_M |\nabla^g \varphi|^2 \vol_g = |\nabla^g \varphi|^2(x) \int_M \vol_g.
\end{equation*}
The identity
\begin{equation*}
\vol_{g} = \vol_{g^A} =\det(A^{-1}) \vol_{\bar g}
\end{equation*}
implies
\begin{equation*}
\int_M \vol_g = \frac{1}{\det(A)} \int_M \vol_{\bar g}=\frac{\vol^{\bar g}(M)}{\det(A)}=\frac{1}{\det(A)}
\end{equation*}
as by assumption $\vol^{\bar g}(M)=1$. Altogether, we obtain the following theorem on the energy of $G$-invariant universal spinor fields:
\begin{theorem}
Suppose $M$ is a reductive homogeneous space and $\bar g$ is a $G$-invariant background metric of unit volume on $M$. If $\Phi = [\tilde A, \varphi] \in \mc{F}^G$, then
\begin{equation}\label{spin_energy_homog}
\mc{E}(\Phi) = \frac{1}{32} \frac{1}{\det (A)} \sum_i \Bigl|\sum_{j<k} c_{ijk}(A) E_j \cdot E_k \cdot \varphi\Bigr|^2.
\end{equation}
\end{theorem}
To discuss also the volume-normalized spinor flow in a homogeneous setting, we need to identify those
$G$-invariant universal spinor fields $[\tilde{A},\varphi]$ of length one for which additionally $\vol^{g^A}(M)=1$.
By the above computations, we see that these are exactly those satisfying $\det(A)=1$, i.e.
we have 
\[
\mc{N}_1^G:=\mc{N}^G\cap \mc{N}_1=\bigl\{[\tilde{A},\varphi]\in \mc{N}^G:\det(A)=1\bigr\}.
\]
Hence, the metric parts of elements in $\mc{N}_1^G$ constitute the subset $\Sym_+^0(n,\bR)^H:=\Sym_+(n,\bR)^H\cap \mathrm{SL}(n,\bR)$
of $\Sym_+(n,\bR)^H$.
\subsection{The homogeneous spinor flows}
Next, we study both spinor flows in our homogeneous setting. Let $\Phi\in \mc{N}^G$
or $\Phi\in \mc{N}_1^G$, respectively, and choose some $\tilde{A}\in \widetilde{\GL}_+(n)$ covering $g_{\Phi}$,
i.e.\@ with $A^{-T}A^{-1}=g_{\Phi}$. Then Proposition \ref{pro:spinorflowspindiffeos} implies that the solution $t\mapsto \Phi_t$ of the
usual or the volume-normalized spinor flow, respectively, with initial value $\Phi$ stays in $\mc{N}^G$ or $\mc{N}_1^G$,
respectively, and so may be identified by Lemma \ref{le:trivialization} with a curve $t\mapsto (g_t,\varphi_t)$
in $\Sym_+(n,\bR)^H\times \Sigma_n^H$ or $\Sym_+^0(n,\bR)^H\times \Sigma_n^H$, respectively.
Moreover, $Q(\Phi)$ or $\tilde{Q}(\Phi)$, respectively, is $G$-invariant by Proposition \ref{pro:spinorflowspindiffeos}. Thus, we have
\begin{equation*}
Q(\Phi)=(Q_1(\Phi),Q_2(\Phi))\in \Sym(n,\bR)^H\times (\varphi^{\perp})^H\subseteq \Sym(n,\bR)^H\times \Sigma_n^H
\end{equation*}
or
\begin{equation*}
\begin{split}
\tilde{Q}(\Phi)=(\tilde{Q}_1(\Phi),\tilde{Q}_2(\Phi))=(\tilde{Q}_1(\Phi),Q_2(\Phi))&\in 
T_{g} \Sym_+^0(n,\bR)^H\times (\varphi^{\perp})^H\\
&\quad \subseteq T_{g} \Sym_+^0(n,\bR)^H\times \Sigma_n^H,
\end{split}
\end{equation*}
respectively, where $H$ acts on $\Sigma_n$ by $\rho_n \circ \tilde A^{-1} \tilde \alpha \tilde A$. Hence, we define:
\begin{definition}
The {\em homogeneous spinor flow} on $M$ is the system of ODEs on $\Sym_+(n,\bR)^H\times \Sigma_n^H$
given by
\begin{equation}\label{eq:homspinflow}
(\dot{g}_t,\dot{\varphi}_t)=(Q_1(g_t,\varphi_t),Q_2(g_t,\varphi_t)).
\end{equation} 
The {\em volume-normalized homogeneous spinor flow} on $M$ is the system of ODEs on $\Sym_+^0(n,\bR)^H\times \Sigma_n^H$
given by
\begin{equation}\label{eq:normalizedhomspinflow}
(\dot{g}_t,\dot{\varphi}_t)=(\tilde{Q}_1(g_t,\varphi_t),\tilde Q_2(g_t,\varphi_t))
\end{equation} 
\end{definition}
For later applications, we discuss how one may compute $Q(\Phi)$ and $\tilde{Q}(\Phi)$ in practice. This discussion
will also allow us to define both homogeneous spinor flows for non-compact reductive homogeneous spaces.

Let us start with $Q(\Phi)$, i.e. $\Phi\in \mc{N}^G$, and more specifically with the computation of $Q_1(\Phi)$. Note that
$Q_1(\Phi)\in \Gamma(\odot^2 T^* M)$ is uniquely defined by $((Q_1(\Phi),\dot{g}))_g=-\left.\tfrac{d}{dt}\right|_{t=0} \cE(\Phi_t)$
for $t\mapsto \Phi_t\in \mathcal{N}$ being a horizontal curve with $\Phi_0=\Phi$ such that the associated
curve $t\mapsto g_t$ of Riemannian metrics fulfills $\dot{g}_0=\dot{g}$. In our setting, everything is $G$-invariant and
so it suffices to consider $\dot{g}\in  \Gamma(\odot^2 T^* M)^G\cong \Sym(n,\bR)^H$. For such a $\dot{g}$, we obtain
\begin{equation*}
\begin{split}
((h,\dot{g}))_g &=\int_M (h,\dot{g})_g\, \vol^g=\frac{1}{\det(A)} \int_M \sum_{i,j=1}^n h(AX_i,AX_j)\, \dot{g}(AX_i,A X_j)\, \vol^{\bar g}\\
&=\frac{1}{\det(A)}\sum_{i,j=1}^n h(AX_i,AX_j)\, \dot{g}(AX_i,A X_j)\\
&=\frac{1}{\det(A)} \sum_{i,j=1}^n (A^T h A)_{ij}\, (A^T \dot{g} A)_{ij}=\frac{1}{\det(A)} \tr(h g^{-1} \dot{g} g^{-1})
\end{split}
\end{equation*}
for any $h\in \Sym(n,\bR)^H\cong (\odot^2 \mfp^*)^H$. Now let $t\mapsto g_t \in \Sym_+(n,\bR)^H\cong (\odot_+^2 \mfp^*)^H$ be a smooth 
curve with $g_0=g$ and $\dot{g}_0=\dot{g}$. Then the horizontal curve $t\mapsto \Phi_t$ from above is given by
$t\mapsto [\tilde{A}_t,\varphi]$ for $(\tilde{A}_t)_{t\in I}$ being the horizontal lift of $(g_t)_t$ with $\tilde{A}_0=\tilde{A}$
and we obtain
\begin{equation}\label{eq:Q1hom}
\tr(Q_1(\Phi) g^{-1} \dot{g} g^{-1})=\det(A)((Q_1(\Phi),\dot{g}))_g=-\det(A)\left.\tfrac{d}{dt}\right|_{t=0} \cE([\tilde{A}_t,\varphi]).
\end{equation}
The computation for $Q_2(\Phi)$ is similar but much easier and one obtains
\begin{equation}\label{eq:Q2hom}
\langle Q_2(\Phi), \dot{\varphi}\rangle =-\det(A)\cdot\left.\tfrac{d}{dt}\right|_{t=0} \cE([\tilde{A},\varphi+t\dot{\varphi}]).
\end{equation}
for any $\dot{\varphi}\in (\varphi^{\perp})^H\subseteq \Sigma_n^H$. For $\Phi\in \mc{N}_1^G$, we have
and $\tilde{Q}_1(g,\varphi)=Q_1(g,\varphi)+\frac{n-2}{2n}\cE(g,\varphi)\, g$ and $\tilde{Q}_2(\Phi)=Q_2(\Phi)$
by equation \eqref{eq:tildeQ}.

At this point, note that all of the above spaces of $H$-invariant tensors and spinors and all of the above explicit formulas for the different quantities related to the homogeneous spinor flows on a compact homogeneous manifolds may also be defined for non-compact (reductive) homogeneous spaces. Hence, we arrive at the following definition:
\begin{definition}
Let $M=G/H$ be a non-compact reductive homogeneous space. Then we define the \emph{spinorial energy functional} by equation \eqref{spin_energy_homog}.\\
Moreover, for $\Phi=[\tilde{A},\varphi]\in \mc{N}^G$, we define
\begin{equation*}
Q(\Phi)=(Q_1(\Phi),Q_2(\Phi))\in \Sym(n,\bR)^H\times \Sigma_n^H
\end{equation*}
by equations \eqref{eq:Q1hom}, \eqref{eq:Q2hom} and then the {\em homogeneous spinor flow} by equation \eqref{eq:homspinflow}.
Furthermore, we set $\mc{N}_1^G:=\{[\tilde{A},\varphi]\in \mc{N}^G:\det(A)=1\}$, define for $\Phi=[\tilde{A},\varphi]\in \mc{N}_1^G$ the quantity
\begin{equation*}
\begin{split}
\tilde{Q}(\Phi)=(\tilde{Q}_1(\Phi),\tilde{Q}_2(\Phi))&\in T_{g} \Sym_+^0(n,\bR)^H\times \Sigma_n^H
\end{split}
\end{equation*}
by equations \eqref{eq:tildeQ} and then the {\em volume-normalized homogeneous spinor flow} by equation 
\eqref{eq:normalizedhomspinflow}.
\end{definition}
\begin{remark}
In the case of a compact reductive homogeneous space $G/H$, we took above a $G$-invariant background metric
$\bar{g}$ of unit volume $\vol(\bar{g})=1$. If we take instead a $G$-invariant background metric $\tilde{g}$
of some other volume $\vol(\tilde{g})=C>0$, then formula \eqref{spin_energy_homog} for the spinorial energy functional
$\cE$ gets multiplied by $C$ (surely, the energy of a $G$-invariant universal spinor is independent of the chosen background
metric but $A\in \GL_+(n)$ with $\Phi=[\tilde{A},\varphi]$ depends on the background metric). However,
this factor cancels out when computing the flow equations for the homogeneous spinor flow. Moreover, for the volume-normalized
flow, now normalized to volume equal to $C$, the additional term for $Q_1$ is $\frac{1}{C}\frac{n-2}{2n}\cE(g,\varphi)\, g$.
Thus, also in this case, the flow equations stay the same. This is why we do not require below that our chosen
background metric has volume equal to one leading to $\cE(\Phi)$ being correct only up to a positive constant multiple but correct flow equations.
\end{remark}
\begin{remark}
For a non-compact homogeneous space $G/H$ it is not clear, whether a solution of the homogeneous spinor flow also solves the partial differential equations derived in \cite{aww1}. However, if $G/H$ admits a cocompact lattice, the following non-rigorous argument can be made to support this hypothesis.

It is presumably possible to extend the existence and uniqueness theory for the spinor flow on closed manifolds to initial values with bounded geometry on any manifold. In that case the spinor flow would be defined by the partial differential equations \ref{eq:spinorflow}. Since the spinor flow is defined by a partial differential equation, it commutes with the action of a covering map, i.e.\@ if $M, \hat M$ are manifolds and $p: \hat M \to M$ is a covering map and $\Phi_t$ is a solution of the spinor flow on $M$, then $p^* \Phi_t$ is a solution of the spinor flow on $\hat M$.

Now presume that $G/H$ is a non-compact homogeneous space and assume that $\Lambda$ is a cocompact lattice. The space $\Lambda \backslash G/H$ is a compact {\em locally homogeneous space}. Denote by $\pi: G/H \to \Lambda \backslash G/H$ the canonical projection. This induces a pullback map
$$\pi^* : \Sigma (\Lambda \backslash G/H)  \to \Sigma G/H.$$
On the other hand, for $G$-invariant universal spinor fields, there is a push forward map
$$\pi_* : \Gamma(\Sigma G/H)^G \to \Gamma(\Sigma(\Lambda \backslash G/H)).$$
The image of $\pi_*$ can be considered to be the {\em locally invariant} sections of the universal spinor bundle. This set of locally invariant sections will be denoted by
$$\Gamma(\Sigma (\Lambda \backslash G/H))^{\operatorname{loc} G} = \pi_* \Gamma(\Sigma(\Lambda \backslash G/H)).$$
Now assume that $\Phi \in \Gamma(\Sigma (\Lambda \backslash G/H))^{\operatorname{loc} G}$ and let $\tilde \Phi = \pi^* \Phi$ be its pullback. Then the spinor flows $\Phi_t$ and $\tilde \Phi_t$ on $\Lambda \backslash G/H$ and $G/H$ respectively satisfy $\pi^* \Phi_t = \tilde \Phi_t$. Since $\tilde \Phi$ is $G$-invariant by definition, the solution $\tilde \Phi_t$ is also $G$-invariant for every $t$. This implies that $\Phi_t$ also remains locally invariant, i.e.\@ $\Phi_t \in \Gamma(\Sigma (\Lambda \backslash G/H))^{\operatorname{loc} G}$. Thus there is a well-defined notion of a {\em locally homogeneous spinor flow}. 

If $\Phi \in \Gamma(\Sigma(\Lambda \backslash G / H))$ is a universal spinor field, which lifts to an invariant universal spinor field $\tilde \Phi$ on $G/H$, then we can consider the spinor flows with initial condition $\Phi$ on $\Lambda \backslash G/H$ and $\tilde \Phi$ on $G/H$ respectively. Denote these spinor flows by $\Phi_t$ and $\tilde \Phi_t$. Then $\tilde \Phi_t = \pi^* \Phi_t$.

The calculations of the spinorial energy functional and gradient performed in the previous sections could be repeated for the set of locally invariant universal spinors $\Gamma(\Sigma (\Lambda \backslash G/H))^{\operatorname{loc} G}$. The calculations at a point would not change, only the terms involving the volume depend on the lattice $\Lambda$. However, the dependence on the volume cancels in the calculation of the negative gradient $(Q_1, Q_2)$, see the previous remark. Thus precisely the same formulas hold for the locally homogeneous spinor flow. The solution of the locally homogeneous flow is a solution of the negative gradient flow of the spinorial energy functional. Since $\Lambda \backslash G/H$ is a compact manifold, this implies that the solution also solves the partial differential equations \eqref{eq:spinorflow}.

Finally, let $\tilde \Phi \in \Gamma(\Sigma(G/H))^G$ and suppose that $\tilde \Phi_t$ is a solution of the {\em homogeneous spinor flow}. Then $\pi_* \tilde \Phi_t$ solves the {\em locally homogeneous spinor flow}, since the defining formulas coincide. On the other hand, then $\pi_* \tilde \Phi_t$ is also a solution of the spinor flow. It follows that $\pi^* \pi_* \tilde \Phi_t = \tilde \Phi_t$ also solves the partial differential equations \eqref{eq:spinorflow}.
\end{remark}
\section{The homogeneous spinor flow in dimension three}
In this section, we consider the spinor flow in dimension three. Dimension three is rather special, because the spinorial component of the solution only moves by parallel translation with respect to the Bourguignon-Gauduchon connection. This is due to the following theorem.
\begin{theorem}\label{th:homspinflowforn=3}
Let $M=G/H$ be a homogeneous manifold of dimension $3$ and let $\Phi=[\tilde{A},\varphi]\in \mc{N}^G$ be a $G$-invariant universal spinor
on $M$. Then
\begin{equation}\label{eq:energyforn=3}
\cE (\Phi)=\frac{1}{32 \det(A)}\sum_{i=1}^3 (c^2_{i12}(A)+c^2_{i13}(A)+c^2_{i23}(A))
\end{equation}
and, consequently, $Q_2(\Phi)=0$. Hence, the solution of the homogeneous spinor flow on $M$ with initial value $\Phi$ is of the form $I \ni t \mapsto [\tilde{A}_t,\varphi] \in \mc{N}^G$ for a horizontal curve $I\ni t \mapsto  \tilde{A}_t\in \widetilde{\GL}_+(3)$ with $\tilde{A}_0=\tilde{A}$ and the same applies for the solution of the homogeneous volume-normalized spinor flow if $\Phi\in \mc{N}_1^G$.
\end{theorem}
\begin{proof}
Equation \eqref{eq:energyforn=3} follows directly from equation \eqref{spin_energy_homog} by noting that for any
 $(j_1,k_1)\neq (j_2,k_2)$ with $1\leq j_1<k_1\leq 3$, $1\leq j_2<k_2\leq 3$ we must have either $j_1=j_2$ or $j_1=k_2$ and so get
$\langle E_{j_1}\cdot E_{k_1}\cdot \varphi, E_{j_2} \cdot E_{k_2}\cdot \varphi\rangle=0$. The other statements then follow directly from 
equations \eqref{eq:Q2hom}, \eqref{eq:tildeQ} and Lemma \ref{le:trivialization}.
\end{proof}
\subsection{Spinor flow on three-dimensional unimodular Lie groups}
In this section, we consider the homogeneous spinor flow on three-dimensional unimodular Lie groups $G$, i.e.\@ we assume that $H=\{e\}$ and that $\tr(\ad(X))=0$ for all $X\in\mfg$. 

Note that by \cite{mil}, the associated simply-connected Lie groups are precisely those which admit cocompact lattices in dimension three. Hence all of them give rise to Thurston geometries, possibly by passing to the maximal geometry associated with the left-invariant metric, see \cite{scott}.

By the Bianchi classification \cite{bianchi}, the possible Lie algebras $\mfg$ are $\su(2)$, $\sL(2,\bR)$, $\mfe(2)$, $\mfe(1,1)$, 
$\mfh_3$ and $\bR^3$, where $\mfe(2)$ and $\mfe(1,1)$ are the Lie algebras of the group of motions of the Euclidean or the Minkowski plane, respectively, and $\mfh_3$ is the three-dimensional
Heisenberg group. Again by the Bianchi classification, we may choose a basis $X_1,X_2,X_3$ of $\mfg$ such that
\begin{equation}\label{eq:standardbasis}
[X_1,X_2]=\epsilon_3 X_3,\quad [X_2,X_3]=\epsilon_1 X_1\quad [X_3,X_1]=\epsilon_2 X_2
\end{equation}
for certain $\epsilon_1,\epsilon_2,\epsilon_3\in \{-1,0,1\}$. We will call such a basis a \emph{standard} basis for $\mfg$. Note that there are $27$ different choices
for $(\epsilon_1,\epsilon_2,\epsilon_3)$ and all of them correspond to one of the six unimodular three-dimensional Lie algebras mentioned above. So there are surely different values of $(\epsilon_1,\epsilon_2,\epsilon_3)$ corresponding
to the same Lie algebra and we restrict here to give one example of $(\epsilon_1,\epsilon_2,\epsilon_3)$ for each of the six three-dimensional unimodular Lie algebras. Note that we will use exactly that triple below when discussing each Lie algebra
individually:

We may choose $(\epsilon_1,\epsilon_2, \epsilon_3)=(1,1,1)$ for $\mfg=\mf{su}(2)$, 
 $(\epsilon_1,\epsilon_2, \epsilon_3)=(-1,1,1)$  for $\mfg=\mf{sl}(2,\bR)$,  $(\epsilon_1,\epsilon_2, \epsilon_3)=(1,1,0)$ for $\mfg=\mf{e}(2)$,  $(\epsilon_1,\epsilon_2, \epsilon_3)=(1,-1,0)$ for $\mfg=\mf{e}(1,1)$,
 $(\epsilon_1,\epsilon_2, \epsilon_3)=(1,0,0)$ for $\mfg=\mf{h}_3$ and $(\epsilon_1,\epsilon_2, \epsilon_3)=(0,0,0)$ for $\mfg=\bR^3$.

Now let us come to the homogeneous spinor flows on these Lie algebras $\mfg$. By Theorem \ref{th:homspinflowforn=3},
only the metric evolves and so we will consider these flows as flows for metrics $g$ on $\mfg$. Moreover,
we will consider both homogeneous spinor flows for initial metrics $g$ which are diagonal 
with respect to a standard basis $(X_1,X_2,X_3)$, i.e. $g=g_{a_1,a_2,a_3}:=\sum_{i=1}^3 a_i\, X^i\otimes X^i$ for certain 
$a_1,a_2,a_3\in (0,\infty)$ and show that then the solutions of both homogeneous spinor flows
stay diagonal with respect to the standard basis $(X_1,X_2,X_3)$. This is no restriction by the next lemma:

\begin{lemma}\label{le:diagonalbases}
Let $g$ be a Riemannian metric on a unimodular three-dimensional Lie algebra $\mfg$. Then $\mfg$ admits a $g$-orthogonal standard basis.
\end{lemma}
\begin{proof}
The assertion is obviously true for $\mfg=\bR^3$.

If $\mfg=\mf{h}_3$, then $\mf{z}(\mfg)=[\mfg,\mfg]$ and $[\mfg,\mfg]$ is one-dimensional. Hence, one may choose orthogonal $X_2,X_3\in [\mfg,\mfg]^{\perp}$. With $X_1:=[X_2,X_3]$ the basis $(X_1,X_2,X_3)$ is orthogonal and standard.

Next, let $\mfg\in \{ \mf{e}(2),\mf{e}(1,1)\}$. Then $[\mfg,\mfg]$ is two-dimensional and abelian. We choose some non-zero element
$X_3\in [\mfg,\mfg]^{\perp}$, set $f:=\ad(X_3)|_{[\mfg,\mfg]}$ and note that $f$ has to be trace-free. It suffices to show that
there is an orthogonal basis $X_1,X_2$ of $[\mfg,\mfg]$ such that $g(X_1,f(X_1))=g(X_2,f(X_2))=0$ as then we may scale
the vectors $X_1,X_2,X_3$ appropriately to obtain an orthogonal standard basis for $\mfg$:

To prove the existence of such an orthogonal basis of $[\mfg,\mfg]$, consider the smooth function 
$F:S^1\rightarrow \bR$, $F(X):=g(X,f(X))^2$ for $X\in S^1:=\left\{Y\in [\mfg,\mfg]:g(Y,Y)=1\right\}\linebreak\subseteq [\mfg,\mfg]$.
This function has a global minimum in some $X_1\in S^1\subset [\mfg,\mfg]$. Let $X_2\in S^1\subseteq[\mfg,\mfg]$ be such
that $(X_1,X_2)$ is an orthonormal basis of $[\mfg,\mfg]$. Then the trace-freeness of $f$ writes as
\begin{equation*}
g(X_1,f(X_1))+g(X_2,f(X_2))=\tr(f)=0
\end{equation*}
and the curve $\gamma:\bR\rightarrow S^1$, $\gamma(t):=\cos(t)X_1+\sin(t)X_2$ fulfills $\gamma(0)=X_1$,
$\dot{\gamma}(0)=X_2$ and $\ddot{\gamma}(0)=-X_1$. So $\tilde{F}:=F\circ \gamma$ has a global minimum in $0$, which implies that
\begin{align*}
0= \;\tilde{F}'(0) =& \;2 g(X_1,f(X_1))\cdot \Bigl(g(X_1,f(X_2))+g(X_2,f(X_1))\Bigr)\\
0\leq \tilde{F}''(0) =& \;2 \Bigl(g(X_1,f(X_2))+g(X_2,f(X_1))\Bigr)^2\\
& \qquad +4\, g(X_1,f(X_1))\cdot\Bigl(g(X_2,f(X_2)-g(X_1,f(X_1)\Bigr)\\
=& \; 2\Bigl(g(X_1,f(X_2))+g(X_2,f(X_1))\Bigr)^2-8\, g(X_1,f(X_1))^2,
\end{align*}
where we applied the trace-freeness of $f$ in the last step. The first equation gives us that
$g(X_1,f(X_1))=0$ or $g(X_1,f(X_2))+g(X_2,f(X_1))=0$. In the first case, the trace-freeness
of $f$ gives us that then also $g(X_2,f(X_2))=0$ and we are done. In the other case, the obtained inequality
from $\tilde{F}''(0)\geq 0$ reduces to $0\leq -8  g(X_1,f(X_1))^2$ and so we again have $g(X_1,f(X_1))=0$
and then $g(X_2,f(X_2))=0$. 

Next, let $\mfg=\su(2)$ and take a standard basis $(X_1,X_2,X_3)$ of $\su(2)$ for $(\epsilon_1,\epsilon_2, \epsilon_3)=(1,1,1)$.
Then the group of inner automorphisms of the Lie algebra $\su(2)$ is given by
$\mathrm{Ad}(\SU(2))=\SU(2)/Z(\SU(2))=\SU(2)/\{\pm 1\}=\mathrm{SO}(3)$ and if one identifies $\su(2)$ with $\bR^3$ via the standard basis
$(X_1,X_2,X_3)$, the inner automorphism group $\mathrm{SO}(3)$ acts on $\bR^3$ via its standard representation.
Hence, the assertion for $\mfg=\su(2)$ follows from the principal axis theorem.

Finally, let $\mfg=\mf{sl}(2,\bR)$ and take now a standard basis of $\mf{sl}(2,\bR)$ for
$(\epsilon_1,\epsilon_2, \epsilon_3)=(-1,1,1)$. If we identify $\mf{sl}(2,\bR)$ with $\bR^3$ via that basis, the Killing form
of $\mf{sl}(2,\bR)$ equals (up to a constant multiple) the Lorentz metric $\langle\cdot\,,\cdot\rangle_{1,2}$ on $\bR^3$. As the group
of inner automorphisms $\mathrm{Ad}(\SL(2,\bR))$ of $\mf{sl}(2,\bR)$ preserves the Killing form, $\mathrm{Ad}(\SL(2,\bR))$ may be identified
with a subgroup of $\mathrm{O}(1,2)$ and as $\mathrm{Ad}(\SL(2,\bR))$ is connected, it is easy to see that actually $\mathrm{Ad}(\SL(2,\bR))\cong \SO(1,2)^+$.
Hence, each (oriented, time-oriented) orthonormal basis of $\bR^{1,2}$ corresponds to a standard basis of $\mf{sl}(2,\bR)$ and so we need to find
such a basis under which $g$ is diagonal to prove the assertion. To that end, we diagonalize $\langle \cdot \,, \cdot \rangle_{1,2}$ with respect to $g$, i.e.\ we choose a $g$-orthonormal basis $(X_1,X_2,X_3)$ which is orthogonal with respect to $\langle \cdot \,, \cdot \rangle_{1,2}$. By suitably rescaling and possibly reordering this basis, we obtain a $g$-orthogonal standard basis of $\mf{sl}(2,\bR)$.
This finishes the proof.
\end{proof}

Now let $(X_1,X_2,X_3)$ be a standard basis for some three-dimensional unimodular Lie algebra $\mfg$ and let $\bar{g}$ the 
left-invariant metric for which $(X_1,X_2,X_3)$ is an orthonormal basis. Use the basis $(X_1,X_2,X_3)$ to identify $\mfg$ 
with $\bR^3$ and note that then
\begin{equation*}
[X,Y]=D (X\times Y)
\end{equation*}
for $X,Y\in \mfg$, where $D:=\diag(\epsilon_1,\epsilon_2,\epsilon_3)$ and $\times$ is the usual cross product on $\bR^3$. Now $(AX)\times (AY)= \det(A)\cdot A^{-T} (X\times Y)$ for any $X,Y\in \bR^3$
and so
\begin{equation*}
\begin{split}
\bar g(A^{-1}[A X,AY],Z)&=\bar g(A^{-1} D(AX\times A Y),Z)=\bar g\left(\det(A)\, A^{-1}DA^{-T}(X\times Y),Z\right)\\
&=\bar g\left(X\times Y,\det(A)\, A^{-1}DA^{-T} Z\right)
\end{split}
\end{equation*}
for all $X,Y,Z\in \mfg$. We set $B:=A^{-1}DA^{-T}$ and obtain
\begin{equation*}
\begin{split}
c_{i\, i+1\, i+2}(A)&=-\det(A)\cdot\left(\bar g(X_{i+2}\times  X_i, B X_{i+1})-\bar g(X_{i+1}\times  X_{i+2}, B X_i)\right.\\
&\quad \left.-\bar g(X_{i+1}\times X_i, B X_{i+2})\right)\\
&=\det(A)\cdot(b_{i\,i}-b_{i+1\, i+1}-b_{i+2\, i+2}) \\
c_{i\, i+2\, i+1}(A)&=-c_{i\, i+1\, i+2}(A) =\det(A)\cdot (b_{i+1\, i+1}+b_{i+2\, i+2}-b_{i\,i})\\
c_{i\, i\, i+1}(A)&=-c_{i\, i+1\, i}(A)=2 b_{i+2\, i}\,\det(A),\\
c_{i\, i\, i+2}(A)&=-c_{i\, i+2\, i}(A)=-2 b_{i+1\, i}\,\det(A).
\end{split}
\end{equation*}
for any $i=1,2,3$, where we compute the indices cyclically. Hence, equation \eqref{eq:energyforn=3} gives us
\begin{equation*}
\begin{split}
\cE (\Phi)&=\frac{1}{32\det(A)}\,\sum_{i=1}^3 (c^2_{i12}(A)+c^2_{i13}(A)+c^2_{i23}(A))\\
&=\frac{\det(A)}{8}\,\sum_{\stackrel{i,j=1,}{i\neq j}}^3 b_{ij}^2 + 3(b_{11}^2+b_{22}^2+b_{33}^2)-2( b_{11} b_{22}+b_{11} b_{33}+b_{22} b_{33})
\end{split}
\end{equation*}
Now let $\Phi$ be such that $g:=\diag(a_1,a_2,a_3)$ for certain $a_1,a_2,a_3\in (0,\infty)$. We may choose
$\tilde{A}$ with $A^{-1}=\diag(\sqrt{a_1},\sqrt{a_2},\sqrt{a_3})$.

We first show that $Q_1(g)$ has to be diagonal, which then implies that $\tilde{Q}_1(g)$ is diagonal as well. To show this, we convince ourselves that for any $\dot{g}\in \Sym^2(3,\bR)$ which is off-diagonal, one has $\left.\tfrac{d}{dt}\right|_{t=0} \cE(\Phi_t)=0$
for $\Phi_t$ being the horizontal lift of a curve $I\ni t\mapsto g_t\in \Sym_+^2(n,\bR)$ with $g_0=g$ and $\dot{g}_0=\dot{g}$.
Without loss of generality,
we may assume that $\dot{g}=\left(\begin{smallmatrix} 0 & 1 & 0 \\ 1 & 0 & 0 \\ 0 & 0 & 0 \end{smallmatrix}\right)$.
The other cases are similar. We set
\begin{equation*}
A_t:=\begin{pmatrix} \tfrac{4 \sqrt{a_1} a_2}{4 a_1 a_2-t^2} & -\tfrac{ 2t \sqrt{a_2}}{4 a_1 a_2-t^2} & 0 \\ -\tfrac{ 2t \sqrt{a_1}}{4 a_1 a_2-t^2} & \tfrac{4 \sqrt{a_2} a_1}{4 a_1 a_2-t^2} & 0 \\
0 & 0 & \tfrac{1}{\sqrt{a_3}} \end{pmatrix}=\begin{pmatrix} \sqrt{a_1} & \tfrac{t}{2\sqrt{a_1}} & 0 \\ \tfrac{t}{2\sqrt{a_2}} &\sqrt{a_2} & 0 \\
0 & 0 & \sqrt{a_3} \end{pmatrix}^{-1}
\end{equation*}
Then one gets that
\begin{equation*}
g_t:=A_t^{-T} A_t^{-1}=\begin{pmatrix} \tfrac{4 a_1 a_2+t^2}{4 a_2} & t & 0 \\ t & \tfrac{4 a_1 a_2+t^2}{4 a_1} & 0 \\
0 & 0 & a_3
 \end{pmatrix}
\end{equation*}
fulfils $g_0=g$ and $\dot{g}_0=\dot{g}$. Moreover, $A_0=A$ and
\begin{equation*}
A_t^{-1}\, \dot{A}_t=\begin{pmatrix} \tfrac{t}{4 a_1 a_2-t^2} & -\tfrac{ 2\sqrt{a_1 a_2}}{4 a_1 a_2-t^2} & 0 \\  -\tfrac{ 2\sqrt{a_1 a_2}}{4 a_1 a_2-t^2} & \tfrac{t}{4 a_1 a_2-t^2} & 0 \\
0 & 0 & 0\end{pmatrix} 
\end{equation*}
i.e. $A_t^{-1}\, \dot{A}_t$ is symmetric for all $t\in I$. Hence, $I\ni t\mapsto A_t\in \GL_+(n)$ is a horizontal lift as desired. Moreover,
if $\Phi_t=[\tilde{A}_t,\varphi]$ is the horizontal lift with $\tilde{A}_0=\tilde{A}$ such that $A_t$ is the image of $\tilde{A_t}$ under the universal cover $\widetilde{\GL}_+(n)\rightarrow \GL_+(n)$,
then we have
\begin{equation*}
B_t:=A_t^{-1} D A_t^{-T}=\begin{pmatrix} \tfrac{4 a_1^2 \epsilon_1+t^2\epsilon_2}{4a_1} & \tfrac{t(a_1 \epsilon_1+a_2 \epsilon_2)}{2 \sqrt{a_1 a_2}}  & 0 \\
\tfrac{t(a_1 \epsilon_1+a_2 \epsilon_2)}{2 \sqrt{a_1 a_2}} & \tfrac{4 a_2^2 \epsilon_2+t^2\epsilon_1 }{4a_2} & 0 \\
0 & 0 & a_3 \epsilon_3
\end{pmatrix}. 
\end{equation*}
and one sees that $\cE (\Phi_t)$ is a rational function in $t$ with no linear terms in the numerator as well as in the denominator.
This implies that $\left.\tfrac{d}{dt}\right|_{t=0} \cE(\Phi_t)=0$ and so that $Q_1(\Phi)$ and $\tilde{Q}_1(\Phi)$
both have no off-diagonal part.

To compute $Q_1(\Phi)_{11}$, note that a horizontal curve $I\ni A_t\in \GL_+$ with $A_0=A$ and $g_t:=A_t^{-T} A_t^{-1}$ fulfilling $\dot{g}_0=g E_{11}g=a_1^2 E_{11}$ is given by
\begin{equation*}
A_t^{-1}:=\diag(\sqrt{(1+a_1 t)a_1},\sqrt{a_2},\sqrt{a_3}).
\end{equation*}
Hence, $B_t=\diag((1+a_1 t)\epsilon_1 a_1 , \epsilon_2 a_2 , \epsilon_3 a_3 )$ and
$\det(A_t)=((1+a_1 t) a_1 a_2 a_3)^{-\tfrac{1}{2}}$.
We set now $b_i:=\epsilon_i a_i$ and obtain $\left.\tfrac{d}{dt}\right|_{t=0} b_{11}^2(t)=2 a_1 b_1^2$, $\left.\tfrac{d}{dt}\right|_{t=0} b_{11}b_{ii}=a_1 b_1 b_i$ for $i=2,3$
and $\left.\tfrac{d}{dt}\right|_{t=0} \det(A_t)=-\tfrac{a_1}{2} \det(A)$ and so
\begin{equation*}
Q_1(\Phi)_{11}=-\det(A)\left.\tfrac{d}{dt}\right|_{t=0} \cE(\Phi_t)=-a_1\tfrac{9 b_1^2-3(b_2^2+b_3^2)+2 b_2 b_3-2(b_1 b_2+b_1 b_3)}{64 a_1 a_2 a_3}
\end{equation*}
by equation \eqref{eq:Q1hom}. The computation of $Q_1(\Phi)_{ii}$ for $i=2,3$ is similar and one obtains that the homogeneous spinor flow is given by the following system of ODEs
\begin{equation*}
\dot{a_i}=-\frac{a_i \left(9 b_i^2-3 (b_{i+1}^2+b_{i+2}^2)
+2 (b_{i+1} b_{i+2} -b_i b_{i+1} -b_i b_{i+2}) \right)}{64 a_1 a_2 a_3},\quad i=1,2,3.
\end{equation*}
Moreover, using equation $\eqref{eq:tildeQ}$, we get that the volume-normalized homogeneous spinor flow is given by the following system 
of ODEs
\begin{equation*}
\dot{a_i}=-\frac{a_i \left(8 b_i^2-4 (b_{i+1}^2+b_{i+2}^2)
+\frac{4}{3} (2 b_{i+1} b_{i+2} -b_i b_{i+1} -b_i b_{i+2}) \right)}{64 a_1 a_2 a_3},\quad i=1,2,3,
\end{equation*}
where we note that this system is, in fact, only two-dimensional as $a_1 a_2 a_3=\det(A)^{-\tfrac{1}{2}}=1$.

\medskip

\noindent We have a closer look at the different possible cases:

\medskip
\begin{itemize}
\item[(i)]
First of all, we consider the cases when at least one $\epsilon_i$ is zero. Of course, if all $\epsilon_i$ are zero, i.e. $\mfg=\bR^3$, then 
all evolutions are trivial.

\medskip

\item[(ii)]
Let us now look at the case that exactly one of the $\epsilon_i$ is not zero. Without loss of generality, we
may assume that $\epsilon_1=1$ and $\epsilon_2=\epsilon_3 = 0$. Then $\mfg$ is the three-dimensional Heisenberg algebra $\mfh_3$, $b_1=a_1$
and $b_2=b_3=0$.
So the homogeneous spinor flow is given by
\begin{equation*}
\dot{a}_1=-\frac{9 a_1^2}{64 a_2 a_3},\qquad \dot{a}_2=\frac{3 a_1}{64 a_3},\qquad \dot{a}_3=\frac{3 a_1}{64 a_2}.
\end{equation*}
For the initial values $a_1(0)=a_1^0$, $a_2(0)=a_2^0$, $a_3(0)=a_3^0$, the solution is given by
\begin{equation*}
\begin{split}
a_1(t)&=a_1^0\, \left(\frac{15\, a_1^0}{64\, a_2^0 a_3^0}\,t+1\right)^{-\tfrac{3}{5}},\qquad
a_2(t)=a_2^0\, \left(\frac{15\, a_1^0}{64\, a_2^0 a_3^0}\,t+1\right)^{\tfrac{1}{5}},\\
a_3(t)&=a_3^0\, \left(\frac{15\, a_1^0}{64\, a_2^0 a_3^0}\, t+1\right)^{\tfrac{1}{5}}.
\end{split}
\end{equation*}
Moreover, the volume-normalized homogeneous spinor flow is given by
\begin{equation*}
\dot{a}_1=-\frac{a_1^2}{8 a_2 a_3},\qquad \dot{a}_2=\frac{a_1}{16 a_3},\qquad \dot{a}_3=\frac{a_1}{16 a_2},
\end{equation*}
whose solution with the initial values $a_1(0)=a_1^0$, $a_2(0)=a_2^0$, $a_3(0)=a_3^0$ fulfilling
$a_1^0 a_2^0 a_3^0=1$ equals
\begin{equation*}
\begin{split}
a_1(t)&=a_1^0\, \left(\frac{a_1^0}{4\, a_2^0 a_3^0}\, t+1\right)^{-\tfrac{1}{2}},\qquad
a_2(t)=a_2^0\, \left(\frac{a_1^0}{4\, a_2^0 a_3^0}\,t+1\right)^{\tfrac{1}{4}},\\
a_3(t)&=a_3^0\, \left(\frac{a_1^0}{4\, a_2^0 a_3^0}\,t+1\right)^{\tfrac{1}{4}}.
\end{split}
\end{equation*}
Note that by Lemma \ref{le:diagonalbases}, we have determined all solutions of both homogeneous spinor flows on $\mfh_3$.
\medskip

\item[(iii)]
Now assume that exactly two $\epsilon_i$ are non-zero. Without loss of generality, we may assume that $\epsilon_1=1$ and $\epsilon_3=0$ 
and set $\epsilon:=\epsilon_2 \in \{-1,1\}$. Note that for $\epsilon=1$, we have $\mfg=\mfe(2)$, whereas for $\epsilon=-1$
we get $\mfg=\mfe(1,1)$. To simplify the equations, we assume that the initial values for $a_1$ and $a_2$ are the same. Then we have 
$a_1\equiv a_2$ in both cases. Setting $x:=a_1=a_2$, $y:=a_3$, the homogeneous spinor flow is given by
\begin{equation*}
\dot{x}=\frac{\epsilon-3}{32}\frac{x}{y},\qquad \dot{y}=\frac{3-\epsilon}{32}
\end{equation*}
and the volume-normalized homogeneous spinor flow is given by
\begin{equation*}
\dot{x}=\frac{\epsilon-3}{48}\frac{x}{y},\qquad \dot{y}=\frac{3-\epsilon}{24}
\end{equation*}
The solution of the homogeneous spinor flow with inital value \mbox{$x(0)=x_0$}, $y(0)=y_0$ is given by
\begin{equation*}
x(t)=\frac{x_0}{1+\frac{3-\epsilon}{32 y_0}\, t},\quad y(t)=y_0+\frac{3-\epsilon}{32}\, t
\end{equation*}
and the solution of the volume-normalized homogeneous spinor flow with the same initial value fulfilling $x_0^2 y_0=1$
is given by
\begin{equation*}
x(t)=x_0\cdot \left(1+\frac{3-\epsilon}{24 y_0} t\right)^{-\frac{1}{2}},\quad y(t)=y_0+\frac{3-\epsilon}{24} t.
\end{equation*}
So in all the considered non-abelian cases here and both for the non-normalized and the normalized spinor flow, the maximal interval of 
existence $(T_-,T_+)$ fulfills $T_->-\infty$ and $T_+=\infty$ and at each boundary point some directions blow up and the others collapse
and the ones which collapse at one point blow up at the other and vice versa.

\medskip

\item[(iv)]
Now we assume that $\epsilon_1=\epsilon_2=\epsilon_3=1$. Then the Lie algebra is $\mfg=\su(2)$. The Lie group $\SU(2)$ is diffeomorphic to $S^3$ and the left invariant metric $a_1 = a_2 = a_3$ corresponds to a round sphere, whereas $a_1 = a_2 \neq a_3$ corresponds to a Berger sphere, with fiber length proportional to $a_3$. Notice that since $\epsilon_i = 1$, we have $b_i=a_i$ for $i=1,2,3$.
We try to find solutions of the homogeneous spinor flow with initial value being a Berger sphere, i.e. $a_1(0)=\frac{\epsilon^2}{4}$, $a_2(0)=a_3(0)=\frac{1}{4}$.
Then the symmetry in the above equations shows that $a_2(t)=a_3(t)$ for all $t\in I$. We set $x:= 4 a_1$ and $y:=4 a_2=4 a_3$ and obtain the following initial
value problem:
\begin{align*}
  \dot{x} & =-\frac{9x^2}{16 y^2}+\frac{1}{4}+\frac{x}{4y}, &  x(0) = \epsilon^2\\
  \dot{y} & =\quad\frac{3x}{16 y}- \frac{y}{4x}, & y(0) = 1
\end{align*}
The spinor flow on Berger spheres has been analyzed by Wittmann in \cite{wittm} by constructing adapted initial spinor fields and calculating the terms $Q_1$ and $Q_2$ explicitly. The system above is exactly the same as Wittmann obtains in \cite[Lemma 8]{wittm} for $\mu=-\frac{a}{4}$ and $a\in \{-2,2\}$ arbitrary.
In particular, we see that for $\epsilon=1$, the solution is given by $x(t)=y(t)=1-\frac{t}{16}$.
Moreover, the volume-normalized spinor flow with initial value as above is given by
\begin{align*}
  \dot{x} & =-\frac{x^2}{2y^2}+\frac{1}{3}+\frac{x}{6y}, &  x(0) = \epsilon^2 \\
  \dot{y} & =\quad\frac{x}{4y}- \frac{1}{12}-\frac{y}{6 x}, &  y(0)=1.
\end{align*}
These are again exactly the equations that Wittmann gets in the proof of his stability result \cite{wittm} and so any solution has $T_+=\infty$ and converges for $t\rightarrow \infty$ and any $\epsilon>0$ to a Killing spinor on $S^3$ by \cite{wittm}.

We come now back now to the case with arbitrary $a_1,a_2,a_3$.
If we linearize then the volume-normalized flow equations at $(a_1,a_2,a_3)=(1,1,1)$, we obtain the system
\[
\begin{pmatrix}
	\dot \alpha_1 \\ \dot \alpha_2 \\ \dot \alpha_3
\end{pmatrix}
= - \frac{5}{48}
\begin{pmatrix}
	2 & -1 & -1 \\
	-1 & 2 & -1 \\
	-1 & -1 & 2
\end{pmatrix}
\begin{pmatrix}
	 \alpha_1 \\ \alpha_2 \\  \alpha_3
\end{pmatrix}.
\]
The coefficient matrix has eigenvalue $0$ corresponding to the scaling direction $(1,1,1)$ and the double eigenvalue $-\frac{5}{16}$ 
orthogonal to it, implying first of all stability of the volume-normalized flow on the space of left-invariant metrics of fixed volume which are diagonal with respect to the basis $X_1,X_2,X_3$ from above around the Killing spinor metric $\bar{g}$. Hence, we have a ball $B_{R}(1,1,1)$ of some radius $R>0$ in $\bR_+^3$ around $(1,1,1)$ such that for each initial value $g_{\tilde{a}_1,\tilde{a}_2,\tilde{a}_3}$ with $(\tilde{a}_1,\tilde{a}_2,\tilde{a}_3)\in B_{R}(1,1,1)$
and $\tilde{a}_1\tilde{a}_2\tilde{a}_3=1$ the volume-normalized homogeneous spinor flow converges for $t\rightarrow \infty$ to the round metric $g_{1,1,1}=\bar{g}$ on $\SU(2)\cong S^3$ (actually, $\bar{g}$ is the round metric only up to a positive multiple).

Let now $g$ be any left-invariant metric of the same volume as $\bar{g}$ in the ball of radius $R$ around $g_{1,1,1}$ in
$\Sym_+(3,\bR)\cong \odot^2_+ \mf{su}(2)$. By Lemma \ref{le:diagonalbases}, we know that there is a standard basis
$(\tilde{X}_1,\tilde{X}_2,\tilde{X}_3)$ for which $g$ is diagonal, i.e.\@ $g=\sum_{i=1}^3  \tilde{a}_i \tilde{X}^i\otimes \tilde{X}^i$.
As $(\tilde{X}_1,\tilde{X}_2,\tilde{X}_3)$ and $(X_1,X_2,X_3)$ are related to each other by an element of $\mathrm{SO}(3)$, the metric 
for which $(\tilde{X}_1,\tilde{X}_2,\tilde{X}_3)$ is orthonormal equals $\bar{g}$ and so still the distance of $g$ to $\bar{g}$ is 
less than $R$, i.e.\@ $(\tilde{a}_1,\tilde{a}_2,\tilde{a}_3)\in B_{R}(1,1,1)\subseteq \bR_+^3$. Hence, by the above, also the 
volume-normalized homogeneous spinor flow with initial value $g$ converges for $t\rightarrow \infty$ to the Killing spinor 
$\bar{g}$, i.e. the Killing spinor metric $\bar{g}$ is a stable point in the space of \emph{all} left-invariant metrics of fixed volume.

\medskip

\item[(v)]
Next, we consider $\epsilon_1=-1$, $\epsilon_2=\epsilon_3=1$. Then $\mfg=\sL(2,\bR)$ and  $b_1=-a_1$, $b_2=a_2$ and $b_3=a_3$. Again,
assume that $a_2=a_3$ and set $x:= 4 a_1$ and $y:=4 a_2=4 a_3$. Then the homogeneous spinor flow is given by
\begin{equation*}
\dot{x}=-\frac{9x^2}{16 y^2}+\frac{1}{4}-\frac{x}{4y},\qquad \dot{y}=\frac{3x}{16 y}- \frac{y}{4x}.
\end{equation*}
If the initial values $x(0)=x_0$, $y(0)=y_0$ fulfill $y_0=\frac{3}{2} x_0$, we get the soliton solution
\begin{equation*}
x(t)=x_0-\frac{1}{6} t,\quad y(t)=\frac{3}{2} x(t)=\frac{3}{2} x_0-\frac{1}{4} t.
\end{equation*}
Note that the soliton is not a twistor, and so also not a Killing spinor as
$\nabla_{e_1} \varphi=\frac{4}{3\sqrt{x_0}} e_1\cdot \varphi$ and $\nabla_{e_j} \varphi=-\frac{1}{3\sqrt{x_0}} e_j\cdot \varphi$ for $j=2,3$. Note that in contrast to the case $\mfg=\mf{su}(2)$, here also the metric part of the soliton solution is not unique. 
In fact, for any standard basis $(\tilde{X}_1,\tilde{X}_2,\tilde{X}_3)$ of $\mathfrak{sl}(2,\bR)$, we get a soliton solution
$\tilde{g}=\tfrac{2}{3}\tilde{X}^1\otimes \tilde{X}^1+\tilde{X}^2\otimes \tilde{X}^2+\tilde{X}^3\otimes \tilde{X}^3$. So we have a
$\mathrm{SO}(1,2)^+$-orbit of these spinor soliton metrics and we get the same metric precisely when we apply an element of the subgroup $\mathrm{SO}(1,2)^+\cap \mathrm{SO}(3)=\{1\}\times \mathrm{SO}(2)\cong \mathrm{SO}(2)$. Thus, there is a two-dimensional space of left-invariant soliton metrics parameterized by the homogeneous space $\mathrm{SO}(1,2)^+/\mathrm{SO}(2)$.

The normalized homogeneous spinor flow is given by
\begin{equation*}
\dot{x}=-\frac{x^2}{2y^2}+\frac{1}{3}-\frac{x}{6y},\qquad \dot{y}=\frac{x}{4 y}+ \frac{1}{12}-\frac{y}{6x}.
\end{equation*}
As the normalization condition reads $xy^2=1$, the system reduces to
\begin{equation*}
\dot{x}=-\left(\frac{x^3}{2}-\frac{1}{3}+\frac{x^{\frac{3}{2}}}{6}\right)=-\frac{1}{6}\left(3 x^3+x^{\frac{3}{2}}-2\right).
\end{equation*}
It is easy to see that $\dot{x} >0$ if $x < (2/3)^{2/3}$, $\dot{x} = 0$ if $x = (2/3)^{2/3}$, and $\dot{x} < 0$ if $x > (2/3)^{2/3}$. This implies that $T_+=\infty$ and
\begin{equation*}
\lim\limits_{t\rightarrow \infty} x(t)=\left(\frac{2}{3}\right)^{\frac{2}{3}}=:x_{\infty},\quad
\lim\limits_{t\rightarrow \infty} y(t)=\frac{1}{x_{\infty}^{\frac{1}{2}}}=\left(\frac{2}{3}\right)^{-\frac{1}{3}}
=\frac{3}{2} x_{\infty}. 
\end{equation*}
So the solution converges for $t\rightarrow \infty$ to the soliton solution.

Moreover, if we linearize the volume normalized flow equation at the critical point $(a_1,a_2,a_3)=\left(\frac{2}{3},1,1\right)$, we obtain the system
\[
\begin{pmatrix}
	\dot \alpha_1 \\ \dot \alpha_2 \\ \dot \alpha_3
\end{pmatrix}
= - \frac{5}{288}
\begin{pmatrix}
	-12 & 4 & 4 \\
	9 & -21 & 15 \\
	9 & 15 & -21
\end{pmatrix}
\begin{pmatrix}
	 \alpha_1 \\ \alpha_2 \\  \alpha_3
\end{pmatrix}.
\]
The coefficient matrix has eigenvalue $0$ corresponding to the scaling direction $\left(\frac{2}{3},1,1\right)$ and the negative 
eigenvalues $-\frac{5}{16},-\frac{5}{8}$, implying stability of the volume normalized flow on the space of metrics of fixed ``volume'' 
which are diagonal with respect to the basis $X_1,X_2,X_3$ from above. Note that by similar arguments as in case (iv) any metric $g$ 
sufficiently near $\bar{g}$ flows via the volume-normalized homogeneous spinor flow for $t\rightarrow \infty$ to a spinor soliton metric 
$\tilde{g}$ with respect to which $g$ is diagonal. However, as pointed out above, $\tilde{g}$ is, in general, different from $\bar{g}$.
\end{itemize}
\section{Spinor flow on almost abelian Lie groups}

In this section, we consider the spinor flow on \emph{almost abelian} Lie groups $G$, i.e. Lie groups having an abelian codimension one 
normal subgroup. \emph{Almost abelian} Lie algebras $\mfg$, i.e. Lie algebras of almost abelian Lie groups, are thus characterized by 
the existence of a codimension one abelian ideal $\mfu$ and the entire Lie bracket is encoded in one endomorphism $f:=\ad(Y)|_{\mfu}$ of 
$\mfu$ for $Y\in \mfg\setminus \mfu$. Note that for any other $Y'\in \mfg\setminus \mfu$ there exists some $\lambda\in \bR^*$ such that $f'=\ad(Y')|_{\mfu}=\lambda f$, namely the unique element $\lambda\in \bR^*$
with $Y'-\lambda Y\in \mfu$. Note that all three-dimensional Lie algebras except the two simple ones are almost abelian and so this 
subsection can be seen as a generalization of certain parts of the last subsection.

Below, we will also be interested in almost abelian Lie algebras of dimension four. There are plenty of examples of these Lie algebras in dimension four: First of all, there are the direct sums of a three-dimensional almost abelian Lie algebra with $\bR$, namely the unimodular ones $\bR^4$, $\mfh_3\oplus \bR$, $\mf{e}(2)\oplus \bR$, $\mf{e}(1,1)\oplus \bR$, the non-unimodular Lie algebra $\mf{aff}(\bR)\oplus \bR^2$, $\mf{aff}(\bR)$ being the two-dimensional Lie algebra of affine motions of the real line, one more single non-unimodular example and two more $1$-parameter families of non-unimodular examples, cf. \cite{bianchi}. Moreover, there are in total three single (including one nilpotent example), one $1$-parameter family and two $2$-parameter families of non-decomposable almost abelian Lie algebras of dimension four, cf. \cite{mu}.

Note further that other similar geometric flows on almost abelian Lie algebras were already studied before in the literature. First of all, Lauret \cite{lau} studied in detail the Laplacian flow on seven-dimensional almost abelian Lie algebras obtaining many interesting results for that flow including that the maximal interval of existence is always of the form $(T_-,\infty)$ with $-\infty <T_-<0$ and a classification of all (semi-)algebraic Laplacian solitons on these Lie algebras. Moreover,
Bagaglini and Fino studied in \cite{bf} the Laplacian coflow on these Lie algebras and showed that, under additional assumptions,
the Laplacian coflow on these Lie algebras is ancient and has finite $T_+$ and obtained new examples of soliton solutions for that flow.

Coming back to the spinor flow, let $\Phi=(g,\varphi)\in F_n$ be a universal spinor field and fix a basis $X_1,\ldots, X_n$ of $\mfg$ 
such that $X_1,\ldots,X_{n-1}$ is a basis of $\mfu$ and $X_n$ is perpendicular to $\mfu$ with respect to $g$.
Denote by $F\in \bR^{(n-1)\times (n-1)}$ the matrix representing $\ad(X_n)|_{\mfu}\in \End(\mfu)$ with respect to this basis and choose
$\bar g$ in such a way that $X_1, \ldots, X_n$ is a orthonormal basis with respect to $\overline g$. Then $g=\left(\begin{smallmatrix} H &  \\ & h\end{smallmatrix}\right)$ for certain $H\in \Sym_+(n-1,\bR)$
and $h\in \bR_+$. Similarly, we may decompose $A^{-1}\in \GL_+(n)$ with $A^{-T}A^{-1}=g$ into $A^{-1}=\left(\begin{smallmatrix} B &  \\ & b\end{smallmatrix}\right)$ for certain
$B\in \bR^{(n-1)\times (n-1)}$ and $b\in \bR$ and obtain the relations $B^T B=H$ and $b^2=h$.

\begin{proposition}
Suppose $G$ is an almost abelian Lie group and let $\mfg$ be the associated almost abelian Lie algebra with codimension one abelian 
ideal $\mfu$. If $\dim G \leq 4$ or $\ad(Y)|_{\mf u} = \lambda \id_{\mf u}$ for some $\lambda\in \bR$ and some $Y\in \mfg\setminus \mfu$
, the spinorial energy functional for any $\Phi \in F_n^G$ depends only on the $G$-invariant metric
$g=\left(\begin{smallmatrix} H &  \\ & h\end{smallmatrix}\right)$ defined by $\Phi$ and the following formula holds:
  \begin{equation}\label{eq:spinenergyalmostabelian}
    \mc{E}(\Phi) = \frac{\sqrt{\det(H)}}{32\, \sqrt{h} } \left(3\tr(H F H^{-1} F^T)+\tr(F^2)\right)
  \end{equation}
\end{proposition}
\begin{proof}
We compute $c_{\alpha\beta\gamma}:=c_{\alpha \beta \gamma}(A)$ for all $\alpha,\beta,\gamma$. First, make the trivial observations
that $c_{ijk}=0$ for all $i,j,k\in \{1,\ldots,n-1\}$ and that $c_{\alpha n n}=c_{n\alpha n}=c_{nn \alpha}=0$ for all $\alpha\in \{1,\ldots,n\}$.
So we are left with those $c_{\alpha\beta\gamma}$ where exactly two indices are from $\{1,\ldots,n-1\}$. We have
\begin{equation*}
\begin{split}
c_{ijn}&=-c_{inj}=g_0(B [B^{-1} X_i,\tfrac{1}{b} X_n],X_j)+g_0(B [B^{-1} X_j,\tfrac{1}{b} X_n],X_i)\\
&=-\frac{g_0(BFB^{-1} X_i, X_j)+g_0(BFB^{-1} X_j,X_i)}{b}=-\frac{(D+D^T)_{ij}}{b}
\end{split}
\end{equation*}
with $D:=BFB^{-1}\in \bR^{(n-1)\times (n-1)}$ and, similarly,
\begin{equation*}
\begin{split}
c_{nij}&=g_0(B [\tfrac{1}{b} X_n,B^{-1} X_j],X_i)-g_0(B [\tfrac{1}{b} X_n,B^{-1} X_i],X_j)=-\frac{(D-D^T)_{ij}}{b}
\end{split}
\end{equation*}
for all $i,j\in\{1,\ldots,n-1\}$. As for $i\in \{1,\ldots,n-1\}$, $c_{i\alpha\beta}\neq 0$ implies $\alpha=n$ or $\beta=n$, formula \eqref{spin_energy_homog} for $\cE (\Phi)$ simplifies in all dimensions to
\begin{equation*}
\cE (\Phi)=\frac{b\, \det(B)}{32}\left(\sum_{i,j=1}^{n-1} c_{ijn}^2 +\left|\sum_{1\leq i<j \leq n-1} c_{nij}^2 E_i\cdot E_j \cdot \varphi\right|^2\right).
\end{equation*}
Now if $n\leq 4$, obviously $n-1\leq 3$ and by the same argument as in the proof of Theorem \ref{th:homspinflowforn=3}, the latter formula simplifies in this case further to
\begin{equation*}
\begin{split}
\cE (\Phi)&=\frac{b\,\det(B)}{32}\left(\sum_{i,j=1}^{n-1} c_{ijn}^2 +\sum_{1\leq i<j\leq n-1} c_{nij}^2\right)\\
&=\frac{\det(B)}{32\,b}\left(\sum_{i,j=1}^{n-1} (d_{ij}+d_{ji})^2 +\sum_{1\leq i<j\leq n-1} (d_{ij}-d_{ji})^2 \right)\\
&=\frac{\det(B)}{32\, b }\left(3\sum_{i,j=1}^{n-1} d_{ij}^2+\sum_{i,j=1}^{n-1} d_{ij} d_{ji}\right)=\frac{\det(B)}{32\, b } \left(3\tr(DD^T)+\tr(D^2)\right)\\
&=\frac{\sqrt{\det(H)}}{32\, \sqrt{h} } \left(3\tr(H F H^{-1} F^T)+\tr(F^2)\right).
\end{split}
\end{equation*}
This simplification also takes place if $F=\lambda\, I_{n-1}$ for some $\lambda\in \bR$ and $n$ arbitrary as then always $D=\lambda\, I_{n-1}$ and the second term vanishes.
\end{proof}
In the following we will study the spinor flow on almost abelian Lie groups under the assumption that $n \leq 4$ or $F = \lambda I_{n-1}$. It is a priori not clear that the spinor flow preserves the class of metrics $g$ which are presented by $g=\left(\begin{smallmatrix} H &  \\ & h\end{smallmatrix}\right)$ with respect to the basis $X_1,\ldots,X_n$.

To see this it suffices to show that for a curve of metrics $(g_t)_{t\in I}$ represented by
$\left(\begin{smallmatrix} H &  v(t) \\ v(t)^T& h\end{smallmatrix}\right)$ with respect to the basis $X_1,\ldots, X_{n}$ 
with $v(0)=0$ we have $0=\left.\frac{d}{dt}\right|_{t=0}\cE (\Phi_t)$ for the horizontal lift $(\Phi_t)_{t\in I}$ of that curve.

So we need to compute $\cE(\Phi_t)$. Let $\bar{g}_t$ be the metric on $\mfg$ with orthonormal
basis $(X_1(t):=X_1,\ldots, X_{n-1}:=X_{n-1}(t), X_n(t):=X_n-H^{-1} v(t))$ and note that we have
$g_t=\diag(H, h- v(t)^T h^{-1} v(t))$ with respect
to that basis. Hence, we may apply formula \eqref{eq:spinenergyalmostabelian} with respect to the background metric $\bar{g}_t$
to compute $\cE(\Phi_t)$. In general, this formula depends on the chosen background metric since the matrix $F$ 
appearing in that formula is the matrix representing the endomorphism $\ad(X_n(t))|_{\mfu}$ with respect to the basis 
$(X_1(t),\ldots,X_{n-1}(t))$. However, in our case $(X_1(t),\ldots,X_{n-1}(t))=(X_1,\ldots,X_n)$ and $\ad(X_n(t))|_{\mfu}=\ad(X_n)|_{\mfu}$ and so $F$ is independent of $t$. As $g_t=\left(\begin{smallmatrix} H & 0 \\ 0 & h- v(t)^T h^{-1} v(t) \end{smallmatrix}\right)$ with respect to the basis $(X_1(t),\ldots,X_n(t))$ and $v(0)=0$, one immediately sees
from formula \eqref{eq:spinenergyalmostabelian} that $\left.\frac{d}{dt}\right|_{t=0}\cE (\Phi_t)=0$.
Hence, the homogeneous spinor flow preserves the class of metrics for which $X_n$ is perpendicular to $\spa{X_1,\ldots,X_{n-1}}$.
The same is then obviously also true for the volume-normalized homogeneous spinor flow, i.e. we have shown:
\begin{lemma}
Let $G$ be an almost abelian Lie group and $\mfg$ be the associated almost abelian Lie algebra with codimension one abelian ideal $\mfu$
and assume that $\dim \mfg \leq 4$ or $\ad(Y)|_{\mf u} = \lambda \id_{\mf u}$ for some $\lambda\in \bR$ and some $Y\in \mfg\setminus \mfu$. Moreover, let $\Phi=(g,\varphi)$ be a $G$-invariant universal spinor field and let $X_1,\ldots,X_n$ be a basis of 
$\mfg$ such that $X_1,\ldots,X_{n-1}$ is a basis of $\mfu$ and $X_n\in \mfu^{\perp}$. Then the solution of both
the homogeneous and the volume-normalized homogeneous spinor flow with initial value $(g,\varphi)$ is given by $(g_t,\varphi)$ for $g_t$ 
being of the form $g_t=\diag(H_t,h_t)$ with respect to the basis $(X_1,\ldots,X_n)$ for $H_t\in \Sym_+(n-1,\bR)$ and $h_t\in\bR_+$.
\end{lemma}
Next, we look at the flow equations. Let $\Phi=(g,\varphi)$ be a left-invariant universal spinor field on $G$
and $(X_1,\ldots,X_n)$ be a basis of $\mfg$ such that $g=\diag(H,h)$ with respect to that basis for certain
$H\in \Sym_+(n-1,\bR)$ and $h\in\bR_+$. Moreover, set $\cE(H,h):=\cE(\Phi)$ with $\cE(\Phi)$ as in
\eqref{eq:spinenergyalmostabelian}. Then $t\mapsto g_t=\diag(H,h+t)$ fulfills $\dot{g}_t=h E_{nn}$ and
$A\in \GL_+(n)$ with $A^{-1}=\diag(\sqrt{H},\sqrt{h})$ is a lift of $g$. Hence, we get
\begin{equation*}
\begin{split}
Q_1(\Phi)_{nn}&=-\frac{h^2}{\sqrt{\det(H)\, h}}\pdiff{\cE}{h}(H,h)=\frac{1}{64}\left(3\tr(HFH^{-1} F^T)+\tr(F^2)\right)\\
&=\frac{\sqrt{h}}{2\sqrt{\det(H)}} \cE(H,h) \geq 0,\\
\tilde{Q}_1(\Phi)_{nn}&=Q_1(\Phi)_{nn}+\frac{n-2}{2n} \frac{\cE(H,h)\, h}{\sqrt{\det(H)\, h}}=\frac{n-1}{32n} \left(3\tr(HFH^{-1} F^T)+\tr(F^2)\right)\\
&=\frac{(n-1)\,\sqrt{h}}{n\,\sqrt{\det(H)}}\, \cE(H,h) \geq 0,
\end{split}
\end{equation*}
from \eqref{eq:Q1hom} and \eqref{eq:tildeQ}. This shows that $I\ni t \mapsto h_t$ is strictly monotonically increasing along any solution $(\Phi_t)_{t\in I}$ of both the usual and the volume-normalized homogeneous spinor flow,
or $(g,\varphi)$ is a critical point of both homogeneous spinor flows. We show now that if
$F$ is not nilpotent, then $Q_1(\Phi)_{nn}\geq C$ for some $C>0$ independent of $H$ and $h$. Note that then the same is true also for $\tilde{Q}_1(\Phi)_{nn}$ and
we have shown that for $F$ being non-nilpotent, we cannot have any spinor soliton. By definition of the Frobenius norm $\|\cdot\|_F$, the identity $\tr(H FH^{-1} F^T)=\bigl\|\sqrt{H} F \sqrt{H}^{-1}\bigr\|^2_F$ holds. Clearly, the complex eigenvalues $\lambda_1,\ldots,\lambda_{n-1}$ of $F$ and $\sqrt{H} F \sqrt{H}^{-1}$ coincide. For the Frobenius norm the following inequality holds
$$\|B\|_F^2 \geq \sum_{i=1}^{n-1} \lambda_i(B)^2,$$
where $\lambda_i(B)$ are the Eigenvalues of $B$. 
%
%
In particular, we conclude
\begin{equation*}
\tr(H FH^{-1} F^T)=\left\|\sqrt{H}F \sqrt{H}^{-1}\right\|^2_F\geq \sum_{i=1}^{n-1} \left|\lambda_{i}\right|^2
\end{equation*}
and so
\begin{equation*}
\begin{split}
Q_1(\Phi)_{nn}&=\frac{3\tr(HFH^{-1} F^T)+\tr(F^2)}{64}\geq \frac{3\sum_{i=1}^{n-1} \left|\lambda_{i}\right|^2 +\sum_{i=1}^{n-1} ((\Re\,\lambda_i)^2-(\Im\,\lambda_i)^2)}{64}\\
&=\frac{\sum_{i=1}^{n-1} \left(\left|\lambda_{i}\right|^2 +(\Re\,\lambda_i)^2\right)}{32}=:C
\end{split}
\end{equation*}
independently of $t$. If $F$ is not nilpotent, then $C>0$ and so, in particular, $(g,\varphi)$ cannot be a critical point of one of the homogeneous spinor flows.

If $F$ is nilpotent, then we have $\tr(F^2)=0$ and so $(g,\varphi)$ is a critical point of the usual and so of both homogeneous spinor 
flows if and only if $0=\tr(H FH^{-1} F^T)=\left\|\sqrt{H} F \sqrt{H}^{-1}\right\|^2_F$, i.e.\@ if and only if $F=0$, i.e. if and only 
if $\mfg$ is abelian. Hence, we have obtained:
\begin{proposition}\label{pro:almostabelianspinorsolitons}
Let $\mfg$ be an almost abelian Lie algebra of dimension at most four or assume that $\ad(X)|_{\mfu}$ acts as a multiple of the identity on a codimension one
abelian ideal $\mfu$ for any $X\in \mfg\setminus \mfu$. If $\mfg$ is not abelian, then any Lie group $G$ with associated Lie algebra $\mfg$ does not possess a left-invariant spinor
soliton.
\end{proposition}
\begin{remark*}
By \cite{fr}, there are almost abelian Lie algebras of dimension seven which are not abelian but possess a left-invariant parallel $\G_2$-structure $\phi$, which then has to induce a flat metric $g_{\phi}$.
So they also admit a $g_{\phi}$-parallel spinor field $\varphi$ and so a critical point of even the usual homogeneous spinor flow. This shows that Proposition \ref{pro:almostabelianspinorsolitons}
is no longer true in higher dimensions.
\end{remark*}
Next, we compute $Q_1(\Phi)_{ij}$ for all $i,j\in \{1,\ldots,n-1\}$. Denote by $E_{ij}$ the matrix whose only non-zero component is at the place $(i,j)$ and is equal to one. Then $t\mapsto g_t=\diag(H+tHE_{ij} H,h)$ fulfills $\dot{g}_t=HE_{ij} H$ and $A=\diag(\sqrt{H},\sqrt{h})^{-1}$ is a lift of $g$. Thus, \eqref{eq:Q1hom} and \eqref{eq:tildeQ} imply
\begin{equation*}
\begin{split}
Q_1(\Phi)_{ij}&=-\frac{1}{\sqrt{\det(H)\,h}}\left.\frac{d}{dt}\right|_{t=0}\cE(H+t H E_{ij} H,h)\\
&=-\frac{1}{64h}\Biggl(\frac{\tr(\mathrm{adj}(H)H E_{ij} H)\, \left(3\tr(H F H^{-1} F^T)+\tr(F^2)\right)}{\det(H)}+\\
&\qquad\qquad\quad  6\tr(H E_{ij} HFH^{-1}F^T)-6\tr(HF E_{ij} F^T)\Biggr)\\
&=-\frac{1}{64h}\Biggl(H_{ij} \left(3\tr(H F H^{-1} F^T)+\tr(F^2)\right)\\
&\qquad\qquad\quad +6 (HFH^{-1} F^T H)_{ij}-6 (F^T H F)_{ij}\Biggr)
\end{split}
\end{equation*}
as $H$ is symmetric. So if we write $Q_1(\Phi)=\diag(P_1(\Phi),Q_1(\Phi)_{nn})$ and $\tilde Q_1(\Phi)=\diag(\tilde P_1(\Phi),\tilde Q_1(\Phi)_{nn})$, we have
\begin{equation*}
\begin{split}
P_1(\Phi)&=-\frac{1}{64\, h}\left( H\left(3\tr(HFH^{-1}F^T)+\tr(F^2)\right)+6 H F H^{-1} F^T H -6 F^T H F\right),\\
\tilde{P}_1(\Phi)&=P_1(\Phi)+\frac{n-2}{2n} \frac{\cE(H,h)}{\sqrt{\det(H)\, h}}\, H\\
&=-\frac{1}{64\, h}\left( \frac{2}{n}\, H\left(3\tr(HFH^{-1}F^T)+\tr(F^2)\right)+6 H F H^{-1} F^T H -6 F^T H F\right).
\end{split}
\end{equation*}
Summarizing, the homogeneous spinor flow in the above setting is given by
\begin{equation*}
\begin{split}
\dot{H}_t&=-\tfrac{1}{64 h_t}\left( H_t \left(3\tr(H_tFH_t^{-1}F^T)+\tr(F^2)\right)+6 H_t F H_t^{-1} F^T H_t -6 F^T H_t F\right)\\
\dot{h}_t&=\frac{3\tr(H_tFH_t^{-1} F^T)+\tr(F^2)}{64},
\end{split}
\end{equation*}
whereas the volume-normalized homogeneous spinor flow is given by the initial value problem
\begin{equation*}
\begin{split}
\dot{H}_t&=-\tfrac{1}{64\, h_t}\left( \tfrac{2}{n}\, H_t\left(3\tr(H_tFH_t^{-1}F^T)+\tr(F^2)\right)+6 H_t F H_t^{-1} F^T H_t -6 F^T H_t F\right),\\
\dot{h}_t&=\tfrac{n-1}{32n} \left(3\tr(H_tFH_t^{-1} F^T)+\tr(F^2)\right)
\end{split}
\end{equation*}
In the case of the usual homogeneous spinor flow, if we set $k_t:=\det(H_t)$, we obtain
\begin{equation*}
\begin{split}
\dot{k}_t&=\tr(\det(H_t) H_t^{-1}\dot{H}_t)\\
&=-k_t\tfrac{\tr(I_{n-1})\left(3\tr(H_tFH_t^{-1}F^T)+\tr(F^2)\right)+6 \tr (F H_t^{-1} F^T H_t) -6 \tr (H_t^{-1}F^T H_t F)}{64 h_t}\\
&=-(n-1) \tfrac{k_t}{h_t} \tfrac{3\tr(H_tFH_t^{-1}F^T)+\tr(F^2)}{64}=-(n-1)\, k_t\, \tfrac{\dot{h}_t}{h_t},
\end{split}
\end{equation*}
giving us
\begin{equation*}
\frac{\det(H_t)}{\det(H_0)} = \left(\frac{h_0}{h_t}\right)^{n-1}.
\end{equation*}
In the case of the volume-normalized homogeneous spinor flow, we directly get 
\begin{equation*}
\frac{\det(H_t)}{\det(H_0)} = \frac{h_0}{h_t}
\end{equation*}
by the ``volume-normalization''. An easy consequence of these identities and the fact that $\dot{h}_t \geq C > 0$ for both homogeneous
spinor flows is the following theorem.
\begin{theorem}
Let $G$ be an almost abelian Lie group, $\mfg$ be the associated almost abelian Lie algebra and $\mfu$ be an abelian ideal of
codimension one. Assume that, for $Y\in\mfg\setminus \mfu$, the endomorphism $\ad(Y)|_{\mf u}$ is not nilpotent and that either $\dim(G)\leq 4$ or $\ad(Y)|_{\mfu}=\lambda \id|_{\mfu}$ for some $\lambda\in\bR$. 

If $(H_t, h_t)_{t \in (T_-, T_+)}$ is the maximal solution of the initial value problem with $T_-, T_+ \in \R \cup \{-\infty, \infty\}$ corresponding to either the usual homogeneous spinor flow or the volume-normalized homogeneous spinor flow, then $T_-$ is necessarily finite, i.e.\@ $T_- \in \R$ and $\lim_{t\searrow T_-} h_t=0$ and $\lim_{t\searrow T_-} \det(H_t)=\infty$.
Moreover, $\lim_{t\nearrow T_+} h_t=\infty$ and $\lim_{t\nearrow T_+} \det(H_t)=0$.
\end{theorem}
\begin{remark*}
It is not known if $T_+$ is finite or not.
\end{remark*}
\begin{example}
Let $F=I_{n-1}$. Then the homogeneous spinor flow reads
\begin{equation*}
\dot{h}_t=\frac{n-1}{16},\qquad \dot{H}_t=-\frac{n-1}{16\, h_t} H_t
\end{equation*}
and its solution is given by
\begin{equation*}
h_t=h_0+\frac{n-1}{16} t,\quad H_t= \frac{h_0}{h_0+\frac{n-1}{16} t} H_0\, .
\end{equation*}
The volume-normalized homogeneous spinor flow reads
\begin{equation*}
\dot{h}_t=\frac{(n-1)^2}{8 n},\qquad \dot{H}_t=-\frac{n-1}{8 n \, h_t} H_t
\end{equation*}
and its solution is given by
\begin{equation*}
h_t=h_0+\frac{(n-1)^2}{8 n} t,\quad H_t= \frac{h_0^{\tfrac{1}{n-1}}}{\Bigl(h_0+\frac{(n-1)^2}{8 n} t\Bigr)^{\tfrac{1}{n-1}}} H_0\, .
\end{equation*}
\end{example}

\section{The flag manifold}
The cone construction of B\"ar \cite{bae2} relates the geometry of spaces equipped with Killing spinor fields to the geometry of spaces with parallel spinor fields. By a theorem of Wang existence of a parallel spinor field implies that the holonomy of the underlying Riemannian manifold must be reduced. In fact, its holonomy must be $\SU(n), \Sp(n), G_2$ or $\Spin(7)$ and its Ricci curvature tensor vanishes. If $(M, g, \varphi)$ is a Riemannian manifold with a Killing spinor field $\varphi$, i.e. $\varphi$ satisfies
$$\nabla^g_X \varphi = \lambda X \cdot \varphi,$$
then the Riemannian cone $CM = (0, \infty) \times M$ with the metric $g_{CM} = dr^2 + r^2 g$ admits a parallel spinor field. The geometry of $(M,g)$ can then be characterized by the holonomy group of $(CM, g_{CM})$. If the holonomy group of $CM$ is $\SU(n)$, then the manifold $(M,g)$ admits a Sasaki--Einstein structure. In the case $\Spin(7)$, $(M,g)$ admits a weak $G_2$ structure. The case of $G_2$ is perhaps the most mysterious of the cases. In that case $(M,g)$ admits a {\em (strictly) nearly K\"ahler} structure. A nearly K\"ahler structure is given by a Riemannian metric $g$ together with an almost complex structure $J$ satisfying
$$(\nabla^g_X J) X = 0$$
for every vector field $X$. Four homogeneous examples of strictly nearly K\"ahler $6$-manifolds are known: $S^6 = G_2 / \SU(3)$, $S^3 \times S^3 = \SU(2)^3 / \Delta \SU(2)$, $\C P^3 = \Sp(2) / U(1) \times \Sp(1)$ and $F_{1,2} = \SU(3)/T^2$. Cort\'es and V\'asquez \cite{corvas} constructed locally homogeneous examples by forming quotients of the nearly K\"ahler $S^3 \times S^3$. Foscolo and Haskins \cite{foha} recently constructed two new cohomogeneity $1$ examples on $S^6$ and $S^3 \times S^3$. At the time of writing, these are all known examples. We will now study the spinorial energy and the spinor flow on the flag manifold $F_{1,2}$. The flag manifold $F_{1,2}$ is the homogeneous $6$-manifold $\SU(3) / T^2$, where $T^2$ is embedded in $\SU(3)$ via
\[
(\theta_1, \theta_2) \mapsto \begin{pmatrix} \theta_1 & 0 & 0 \\ 0 & \theta_2 & 0 \\ 0 & 0 & \bar \theta_1 \bar \theta_2 \end{pmatrix}
\]
Let $\mathfrak{t}^2$ be the Lie algebra of $T^2$ and let $X_i \in \mf{t}^2$, $i=1,2$, be the standard basis corresponding to $\theta_1$ and $\theta_2$ respectively. The image of $\mathfrak{t}^2$ in the Lie algebra $\su(3)$ of $\SU(3)$ is spanned by the matrices
\[
\begin{pmatrix} i & 0 & 0 \\ 0 & 0 & 0 \\ 0 & 0 & -i \end{pmatrix} \text{ and } \begin{pmatrix} 0 & 0 & 0 \\ 0 & i & 0 \\ 0 & 0 & -i \end{pmatrix}.
\]
A reductive complement $\mathfrak{p}$ is spanned by the matrices $R_1, \ldots, R_6$
\begin{center}
\begin{tabular}{lll}
  $\begin{pmatrix} 0 & 0 & 0 \\ 0 & 0 & -1 \\ 0 & 1 & 0 \end{pmatrix},$ & $-i \begin{pmatrix} 0 & 0 & 0 \\ 0 & 0 & 1 \\ 0 & 1 & 0 \end{pmatrix},$ & $\begin{pmatrix} 0 & 0 & 1 \\ 0 & 0 & 0 \\ -1 & 0 & 0 \end{pmatrix}$, \\
  $-i \begin{pmatrix} 0 & 0 & 1 \\ 0 & 0 & 0 \\ 1 & 0 & 0 \end{pmatrix}$, & $\begin{pmatrix} 0 & -1 & 0 \\ 1 & 0 & 0\\ 0 & 0 & 0 \end{pmatrix}$, & $-i \begin{pmatrix} 0 & 1 & 0 \\ 1 & 0 & 0 \\ 0 & 0 & 0 \end{pmatrix}$.
\end{tabular}
\end{center}
Any $\mf{t}^2$-invariant inner product on $\mf p$ induces a Riemannian metric on $F_{1,2}$.  Consider the inner product on $\mf p$ for which $\nu R_1, \ldots, \nu R_6$ is orthonormal for any $\nu \in \R_+$. There is a unique value of $\nu_1$ for which the volume of $F_{1,2}$ with respect to the induced Riemannian metric is $1$. Denote by $E_1, \ldots, E_6$ the basis $\nu_1 R_1, \ldots, \nu_1 R_6$ and by $\bar g$ the corresponding inner product on $\mf p$. Denote by $e^1, \ldots , e^6$ the dual basis of $E_1, \ldots, E_6$. Computing the linearized isotropy action with respect to the basis $(E_1, \ldots, E_6)$ yields
\[
  \alpha_*(X_1) =
  \begin{pmatrix}
    0 & -1 &  0 & 0 & 0 &  0 \\
    1 &  0 &  0 & 0 & 0 &  0 \\
    0 &  0 &  0 & 2 & 0 &  0 \\
    0 &  0 & -2 & 0 & 0 &  0 \\
    0 &  0 &  0 & 0 & 0 & -1 \\
    0 &  0 &  0 & 0 & 1 &  0
  \end{pmatrix},
  \quad
  \alpha_*(X_2) =
  \begin{pmatrix}
    0 & -2 &  0 & 0 &  0 & 0 \\
    2 &  0 &  0 & 0 &  0 & 0 \\
    0 &  0 &  0 & 1 &  0 & 0 \\
    0 &  0 & -1 & 0 &  0 & 0 \\
    0 &  0 &  0 & 0 &  0 & 1 \\
    0 &  0 &  0 & 0 & -1 & 0
  \end{pmatrix}.
\]

The space of $\mathfrak{t}^2$ invariant bilinear forms $\left(\odot^2 \mathfrak{p}\right)^{\mathfrak t^2}$ is given by the span of
$$\alpha_1 = e_1^2 + e_2^2, \quad \alpha_2 =  e_3^2 + e_4^2, \quad \alpha_3 =  e_5^2 + e_6^2,$$
as has been computed in \cite{cs}, section 5.
Hence the space of invariant metrics is given by
$$\left(\odot^2_+ \mathfrak{p} \right)^{\mathfrak t^2} = \{a_1 \alpha_1 + a_2 \alpha_2 + a_3 \alpha_3 : a_i \in (0, \infty) \}.$$
The metric $g = a_1 \alpha_1 + a_2 \alpha_2 + a_3 \alpha_3$ can be written as
$$g(v,w) = \bar g(A^{-1} v, A^{-1} w)$$
with $A^{-1} = \diag ( \sqrt{a_1}, \sqrt{a_1}, \sqrt{a_2}, \sqrt{a_2}, \sqrt{a_3}, \sqrt{a_3})$. Notice that thus $A^{-1} \alpha_{*} A = \alpha_{*}$ and hence the representation $\rho_n \circ \tilde A^{-1} \tilde \alpha \tilde A$ is independent of the choice of metric. Hence the module of invariant spinors $\Sigma_n^{T^2}$ does not depend on the choice of metric.

It is well known that the spinor module $\Sigma_6$ can be written as
$$\Sigma_6 \cong \C^2 \otimes \C^2 \otimes \C^2.$$
We identify the standard basis of the tensor product with the standard basis of $\C^8$ via
$$e_i \otimes e_j \otimes e_k \leftrightarrow e_{1 + i + 2j + 4k},$$
where the standard basis of $\C^2$ is labelled $e_0, e_1$. The Clifford action of the standard vectors $e_1, \ldots, e_6 \in \R^6$ is then given by
\[
  g_1 \otimes T \otimes T, \;  g_2 \otimes T \otimes T, \quad
  E \otimes g_1 \otimes T, \;  E \otimes g_2 \otimes T, \quad
  E \otimes E \otimes g_1, \;  E \otimes E \otimes g_2,
\]
where
\[
  E = \begin{pmatrix} 1 & 0 \\ 0 & 1 \end{pmatrix}, \quad T = \begin{pmatrix} 0 & -i \\ i & 0 \end{pmatrix}, \quad
  g_1 = \begin{pmatrix} i & 0 \\ 0 & -i\end{pmatrix}, \quad g_2 = \begin{pmatrix} 0 & i \\ i & 0 \end{pmatrix}.
\]
For details on this construction, see \cite{fri}, section 1.3. Lifting the action of $\mf{t}^2$ from $\so(\mf{p})$ to $\spin(\mf{p})$ via the canonical isomorphism
$$\spin(n) \to \so(n)$$
$$e_i \cdot e_j \cdot \mapsto 2 E_{ij},$$
we obtain an explicit representation of the infinitesimal action of $\mf{t}^2$ on $\C^8$. It is then a simple matter to calculate the space of invariant spinors and we obtain the two-dimensional space
$$\Sigma_6^{\mathfrak{t}^2} = \operatorname{span} \{ \varphi_1, \varphi_2 \}$$
where
$$\varphi_1 = \frac{1}{2} (1, 0, 0, -1, 0, -1, -1, 0) \text{ and } \varphi_2 =  \frac{1}{2} (0, 1, 1, 0, 1, 0, 0, -1).$$
Note that $\varphi_1$ and $\varphi_2$ have unit length.

Now suppose that $\Phi$ in $\mc{F}^G$ is given by $g = a_1 \alpha_1 + a_2 \alpha_2 + a_3 \alpha_3$ and $\varphi = \mu_1 \varphi_1 + \mu_2 \varphi_2$ with $|\mu_1|^2 + |\mu_2|^2 = 1$. Evaluating the formula \ref{spin_energy_homog} yields
$$\mc{E}(g, \varphi) = \frac{3}{16} \left( a_1^2 + a_2^2 + a_3^2 \right) - \frac{1}{8} \left( a_2 a_3 + a_1 a_2 + a_1 a_3 \right).$$
(The formidable number of computations necessary to evaluate the formula have been performed by a computer algebra system.) The energy does not depend on the choice of spinor. This is explained by the fact that the two invariant spinors $\varphi_1$ and $\varphi_2$ differ only by the action of the complex volume element, i.e.\@ $\varphi_2 = \omega_{\C} \cdot \varphi_1$. Thus we can once more consider $\mc{E}$ to be a map on the space of invariant metrics alone. The scaling law $\mc{E}( \lambda^2 g) = \lambda^4 \mc{E}(g)$ implies that $\mc{E}$ has no critical points on the positive cone $\R^+ \langle \alpha_1, \alpha_2, \alpha_3 \rangle$. The volume constraint $\int_M \vol_g = \int_M \vol_{\bar g}$ is equivalent to the  condition $a_1 a_2 a_3 = 1$. By the methods of Lagrange multipliers one can show that the only critical point of $\mc{E}$ under this volume constraint is given by $a_1 = a_2 = a_3$. This critical point corresponds to the strictly nearly K\"ahler metric on $F_{1,2}$ with the Killing spinor.

To derive the spinor flow equations, we compute the negative gradient. Because the energy is independent of the spinorial component we again have $Q_2 = 0$. The metric component $Q_1$ is of the form $v_1 \alpha_1 + v_2 \alpha_2 + v_3 \alpha_3$. By calculations similar to the ones in the case of Lie groups we obtain
$$v_i = -\frac{a_i^2}{2 a_1 a_2 a_3} \frac{\partial \mc{E}}{\partial a_i}.$$
Choosing any $\{i, j, k\} = \{1,2,3\}$, this results in
$$v_i = -\frac{1}{16} \frac{a_i}{a_j a_k} \left( 3 a_i -  a_j -  a_k \right).$$
Thus the homogeneous spinor flow is given by the system
\begin{align*}
  \dot{a}_1 & = -\frac{1}{16} \left( \frac{3 a_1^2}{a_2 a_3} - \frac{ a_1}{a_3} - \frac{a_1}{a_2} \right) \\ 
  \dot{a}_2 & = -\frac{1}{16} \left( \frac{3 a_2^2}{a_1 a_3} - \frac{a_2}{a_1} - \frac{a_2}{a_3} \right) \\ 
  \dot{a}_3 & = -\frac{1}{16} \left( \frac{3 a_3^2}{a_1 a_2} - \frac{a_3}{a_1} - \frac{2 a_3}{a_2} \right).
\end{align*}
For the volume normalized spinor flow we compute $\tilde{Q}_1(g, \varphi) = \tilde{v}_1 \alpha_1 + \tilde{v}_2 \alpha_2 + \tilde{v}_3 \alpha_3$ using formula \ref{eq:tildeQ}:
$$\tilde{v}_i = v_i + \frac{1}{a_1 a_2 a_3} \left(\frac{1}{16} \left( a_1^2 + a_2^2 + a_3^2 \right) - \frac{1}{24} \left( a_2 a_3 + a_1 a_2 + a_1 a_3 \right)\right) a_i.$$
After simplification, the volume normalized homogeneos spinor flow is given by the system
\begin{align*}
  \dot{a}_1 & = \frac{1}{48} \left(-2 - 6 \frac{a_1^2}{a_2a_3}  + \frac{a_1(a_2 + a_3)}{a_2 a_3} + 3 \frac{a_2^2 + a_3^2}{a_2 a_3} \right) \\ 
  \dot{a}_2 & = \frac{1}{48} \left(-2 - 6 \frac{a_2^2}{a_1a_3}  + \frac{a_2(a_1 + a_3)}{a_1 a_3} + 3 \frac{a_1^2 + a_3^2}{a_1 a_3} \right) \\ 
  \dot{a}_3 & = \frac{1}{48} \left(-2 - 6 \frac{a_3^2}{a_1a_2}  + \frac{a_3(a_1 + a_2)}{a_1 a_2} + 3 \frac{a_1^2 + a_2^2}{a_1 a_2} \right).
\end{align*}
Linearizing this system at the critical point $(1,1,1)$ yields the linear system
\[
  \dot{x} = \frac{1}{48} \begin{pmatrix} -10 & 5 & 5 \\ 5 & -10 & 5 \\ 5 & 5 & -10 \end{pmatrix} x.
\]
The eigenvalues of the defining matrix are $0$ and $-5/16$. The eigenvalue $0$ is simple and corresponds to the volume normalization. Since the other eigenvalues are negative, we conclude that the strictly nearly K\"ahler metric on $F_{1,2}$ is a stable critical point of the volume normalized homogeneous spinor flow.

The global dynamics of the volume normalized homogeneous spinor flow can be understood by computing the restriction of $\mc{E}$ to the set of invariant metrics with volume $1$. Such a metric is defined by $g = a_1 \alpha_1 + a_2 \alpha_2 + a_3 \alpha_3$, where $a_1 a_2 a_3 = 1$ and $a_1, a_2, a_3 > 0$. We can parametrize this set by $u,v > 0$ by taking $a_1 = u, a_2 = v$ and $a_3 = \frac{1}{uv}$. The spinorial energy functional then becomes a function on $\R_+^2$ and it is given by
$$\mc{E}(g) = \frac{3}{16} \left( u^2 + v^2 + \frac{1}{u^2 v^2} \right) - \frac{1}{8} \left( u v + \frac{1}{u} + \frac{1}{v} \right).$$
The only critical point on $\R_+^2$ is $u = v = 1$, corresponding to the critical point $a_1 = a_2 = a_3 = 1$ we found before. In fact, this critical point is a global minimum, since $\mc{E}$ diverges as $u$ or $v$ get close to $0$ or $\infty$. This implies that the volume normalized homogeneous spinor flow will converge towards the strictly nearly K\"ahler $F_{1,2}$ from any initial condition.

\begin{remark}
In all situations encountered so far the spinorial energy turned out to be independent of the spinorial part of a $G$-invariant universal spinor field, i.e.\ depended on the metric only. One could be tempted to believe that this is a general feature of the homogeneous spinor flow. However, this turns out not to be the case. For the Aloff-Wallach spaces
\[
N_{k,l}= \SU(3)/\mathrm{U}(1)_{k,l}
\] 
where 
\[
\mathrm{U}(1)_{k,l}= \Biggl\{ \begin{pmatrix} e^{ik\theta} & 0 & 0\\0 &  e^{il\theta} &0 \\ 0&0 & e^{-i(k+l)\theta} \end{pmatrix} : \theta \in \R \Biggr\}
\]
we found a true dependence of the energy on the spinorial part. In fact, by work of Reidegeld \cite{reide}, there is a 4-dimensional space of invariant metrics and we found a 2-dimensional space of invariant spinors, for generic values of $k$ and $l$. Unfortunately, we weren't able to further analyze the flow equations so we refrain from reproducing them here. 

\end{remark}
\medskip\noindent
{{\bf Acknowledgements.} The first author was partly supported by a \emph{For\-schungs\-stipendium} (FR 3473/2-1) from the Deutsche Forschungsgemeinschaft (DFG).

\end{document}